\documentclass[letterpaper,11pt]{article}
\usepackage{color}
\usepackage[usenames,dvipsnames]{xcolor}
\usepackage[small]{titlesec}
\usepackage{theorem,amsmath,amssymb,amscd, comment}
\usepackage[all,cmtip]{xy}
\usepackage{booktabs}
\usepackage[hyperfootnotes=false, colorlinks, linkcolor={blue}, citecolor={magenta}, filecolor={blue}, urlcolor={blue}]{hyperref}
\usepackage[overload]{textcase} 

\theoremstyle{change}
\allowdisplaybreaks
\newcommand{\AI}{{\mathcal{AI}}}

\newcommand{\A}{{\mathbb A}}
\newcommand{\Q}{{\mathbb Q}}
\newcommand{\Z}{{\mathbb Z}}
\newcommand{\R}{{\mathbb R}}
\newcommand{\C}{{\mathbb C}}
\newcommand{\bs}{\backslash}

\newcommand{\p}{\mathfrak p}
\newcommand{\eps}{\epsilon}
\newcommand{\F}{\mathfrak F}

\newcommand{\OF}{{\mathfrak o}}
\newcommand{\GL}{{\rm GL}}

\newcommand{\f}{{\rm f}}
\newcommand{\st}{{\rm st}}
\newcommand{\SL}{{\rm SL}}
\newcommand{\SO}{{\rm SO}}

\newcommand{\GSp}{{\rm GSp}}

\newcommand{\Sp}{{\rm Sp}}

\newcommand{\Ad}{{\rm Ad}}

\newcommand{\K}[1]{{\rm K}(\p^{#1})}
\newcommand{\dK}[1]{\tilde{\mathrm K}(\p^{#1})}
\newcommand{\Kl}[1]{{\rm Kl}(\p^{#1})}

\newcommand{\vol}{{\rm vol}}

\newcommand{\Ind}{{\rm Ind}}
\newcommand{\ind}{{\rm ind}}
\newcommand{\cInd}{\text{c-Ind}}

\newcommand{\Hom}{{\rm Hom}}

\newcommand{\mat}[4]{\begin{bsmallmatrix}#1&#2\\#3&#4\end{bsmallmatrix}}
\newcommand{\qed}{\hspace*{\fill}\rule{1ex}{1ex}}
\newcommand{\forget}[1]{}
\def\qdots{\mathinner{\mkern1mu\raise0pt\vbox{\kern7pt\hbox{.}}\mkern2mu
\raise3.4pt\hbox{.}\mkern2mu\raise7pt\hbox{.}\mkern1mu}}

\newenvironment{proof}{\vspace{1ex}\noindent\emph{Proof.}\hspace{0.5em}}
	{\hfill\qed\vspace{2ex}}
\newenvironment{bsmallmatrix}{\left[\begin{smallmatrix}}{\end{smallmatrix}\right]}

\newtheorem{lemma}{Lemma.}[section]
\newtheorem{theorem}[lemma]{Theorem.}
\newtheorem{corollary}[lemma]{Corollary.}
\newtheorem{proposition}[lemma]{Proposition.}

\newtheorem{remark}[lemma]{Remark.}
\newtheorem{conjecture}[lemma]{Conjecture.}

\begin{document}

\thispagestyle{empty}
\begin{center}
 {\bf\Large Simple supercuspidal representations of $\GSp_4$ and test vectors}

 \vspace{3ex}
 Ameya Pitale, Abhishek Saha, Ralf Schmidt
\end{center}

\abstract{We consider simple supercuspidal representations of $\GSp_4$ over a $p$-adic field and show that they have conductor exponent 5. We study (paramodular) newvectors and minimal vectors in these representations, obtain formulas for their matrix coefficients, and compute key local integrals involving these as test vectors.

Our local computations lead to several explicit global period formulas involving automorphic representations $\pi$ of $\GSp_4(\A)$ whose local components (at ramified primes) are simple supercuspidal representations, and where the global test vectors are chosen to be (diagonal shifts of) newforms or automorphic forms of minimal type.  As an analytic application of our work to the sup-norm problem, we show the existence of paramodular newforms on $\GSp_4(\A)$  of conductor $p^5$ that take ``large values" on a fixed compact set as $p\rightarrow \infty$.
}
%

\section{Introduction}\label{s:intro}
\subsection{Motivation}
Period formulas play an important role in the analytic and arithmetic theory of automorphic $L$-functions and have applications to several important problems in analytic number theory and quantum chaos. Given a cuspidal automorphic representation $\pi\simeq\otimes \pi_v$ of $G(\A)$ for some reductive group $G$, a key input in such a period formula is the choice of an automorphic form~$\phi$ in the space of $\pi$ corresponding to a pure tensor~$\otimes \phi_v$. To ensure that the global test vector~$\phi$ is suitable for the application at hand, one needs to choose the local test vectors $\phi_v$  carefully so that they have the necessary invariance properties and such that the corresponding local integrals appearing in the period formula are non-zero and well-controlled.
Consider for example the famous QUE
theorem proved by Lindenstrauss \cite{Lin06}, which states that as $\pi$ traverses a sequence for
which $\pi_\infty$ belongs to the principal series and $\pi_p$ is unramified at all finite primes, the
$L^2$-masses of the spherical vectors $\phi\in\pi$ corresponding to $\otimes_v \phi_v$ equidistribute. A key step in Lindenstrauss' proof of the QUE theorem involves replacing $\phi_\infty$ by a particular vector $\tilde\phi_\infty$
(the microlocal lift) at the archimedean place whose limit
measures acquire additional invariance.  Further illustration of
this principle is given by recent breakthroughs in period-based approaches to the
subconvexity problem, which depend crucially on the construction of good analytic test vectors (see e.g.\ \cite{michel-2009, nelson20, nelson21, hmn22}).

In the simplest and best-studied case of $G=\GL_2$, and $\pi_p$ a ramified representation (of $\GL_2$ over a $p$-adic field), it has been traditional  to take the local test vector $\phi_p$  to be the newvector. But for certain applications, other reasonable choices are often  more useful and
more natural. The paper \cite{HNS} considered the case when $\pi_p$ is supercuspidal and introduced a different choice of local test vector $\phi_p$, called the minimal vector, which is implicit in the
type theory approach to the construction of supercuspidal
representations. The minimal vectors for $\GL_2$ have several remarkable properties which make them good test vectors for key problems. In the last few years, this has led to an emerging theory surrounding applications of minimal vectors to the sup-norm problem \cite{HNS, sup-norm-minimal-compact}, the subconvexity problem \cite{HN-minimal}, the Kuznetsov formula \cite{Hu20}, and explicit Gross--Zagier formulas with applications to the BSD and Sylvester conjectures \cite{HSY19}.

For  higher rank groups $G$ such as $\GL_n$ and $\GSp_{2n}$, it is therefore of significant interest to find local test vectors inside ramified representations of $G$ over a $p$-adic field. A good analytic theory of such test vectors is currently lacking, even for relatively low rank cases of $G$ beyond the basic case  $G=\GL_2$. In this paper, we focus on  $G=\GSp_4$.  The Iwahori-spherical representations of $\GSp_4$ (these can be thought of as corresponding to automorphic forms of squarefree level) are well-understood and there has been a fair bit of work done \cite{Sch02, pitale-bessel, DPSS15}  on test vectors for these representations. However, our understanding of suitable test vectors is extremely limited for more ramified representations such as the supercuspidal ones. For ramified generic representations on $\GSp_4$, the theory of paramodular newvectors was developed in \cite{NF}. These paramodular newvectors  exist for generic representations of conductor $\p^n$ for each $n$ and have good uniqueness properties. However, many key local integrals, such as the local Whittaker integral appearing in the work of Lapid--Mao, have not been computed for the paramodular newvector for any $n>1$. Furthermore there are other important local integrals (such as those occuring in Gan--Gross--Prasad period formulas) for which the paramodular newvector does not even appear to be a good choice. It is therefore important to explore various choices of test vectors for $\GSp_4$ and crucially to compute the corresponding local integrals occurring in global period formulas.

This work is a first step towards the above-described goal. We define and study a particular class of supercuspidal representations of $\GSp_4$ over a $p$-adic field, known as simple supercuspidal representations. These representations are of great interest because of their ease of access, being induced from characters. We hope that the present work can serve as a stepping stone towards a more general theory.  We build a theory of minimal vectors inside these simple supercuspidal representations analogous to the one for $\GL_2$ constructed in \cite{HNS, HN-minimal}. We express the paramodular newvector in terms of the minimal vector and we compute key local integrals (with either the minimal vector or the paramodular newvector as our input test vector) leading to several explicit global period formulas. We also give an application of one of these period formulas  to the sup-norm problem in the level aspect.

In the rest of this introduction, we describe our results in more detail.

\subsection{Simple supercuspidals and local results}

Simple supercuspidal representations were originally defined for a class of simple groups over $p$-adic fields by Gross and Reeder \cite{GrossReeder2010}. In a sense  they provide the easiest construction of supercuspidals, being induced from affine generic characters of a pro-unipotent radical of a maximal compact subgroup. Knightly and Li \cite{KnightlyLi2015} extended the theory of simple supercuspidals to the case of $\GL_n$.

In Sect.~\ref{simplescsec} of this paper we develop the theory of simple supercuspidals of $\GSp_4(F)$, where $F$ is a $p$-adic field, following the theory outlined in the above papers. We construct these representations via compact induction from a character of the group $ZK'$ where $Z$ is the center and \begin{equation}
 K':=\GSp_4(F)\cap\begin{bsmallmatrix}1+\p&\OF&\OF&\OF\\\p&1+\p&\OF&\OF\\\p&\p&1+\p&\OF\\\p&\p&\p&1+\p\end{bsmallmatrix}
\end{equation}
is the pro-unipotent radical of the Iwahori subgroup. We show that these representations have conductor exponent 5, and we study two key distinguished vectors in these representations, namely the minimal vector and the newvector. In Proposition \ref{newminprop}, we write down explicitly the newvector as a linear combination of translates of the minimal vector.

We continue the local theory in Sect.~\ref{s:localprelims} where we study matrix coefficients for these vectors. A striking feature is that the matrix
coefficient associated to an $L^2$-normalized
minimal vector is a character of the supporting
subgroup $ZK'$. This allows us to explicitly compute the formal degree of a supercuspidal representation, which is done in Proposition~\ref{p:formaldeg}. The matrix coefficient of a newvector is much more complicated, but we are able to write down a reasonably explicit formula for their evaluation on the unipotent radical. Interestingly, this formula involves certain sums of hyper-Kloosterman type (see Remark~\ref{rem:kloosterman-type}) which suggests that there is substantial arithmetic information encoded by these matrix coefficients. We go on to compute various local integrals of Novodvorsky, Whittaker and Gan--Gross--Prasad type using minimal vectors or newvectors as test vectors.

\subsection{The Novodvorsky integral representation for the spinor \texorpdfstring{$L$}{}-function}
Let $\A$ be the ring of adeles of $\Q$. Let $\pi\simeq\otimes_v \pi_v$ be a globally generic, cuspidal automorphic representation of $\GSp_4(\A)$ with trivial central character. One can attach to $\pi$ the spinor (degree 4) $L$-function $L(s, \pi)$ which is equal to the $L$-function of the $\GL_4$ automorphic representation obtained by functorial transfer \cite{AS} of $\pi$ from $\GSp_4$ to $\GL_4$. More generally, one can twist by a Dirichlet character $\chi$ and define the $L$-function $L(s, \pi \times \chi)$. A key tool to understand these $L$-functions is an integral representation provided by Novodvorsky~\cite{nov79} (see also Bump~\cite{Bump87} and Takloo-Bighash~\cite{Ram-Tak}).

The above integral representation involves a global integral $Z(s, \phi, \chi)$, depending on a choice of automorphic form $\phi$ in the space of $\pi$ corresponding to a pure tensor $\otimes_v \phi_v$. The global integral $Z(s, \phi, \chi)$ factors into a product of local integrals $Z(s, W_{\phi_v}, \chi_v)$ where $W_{\phi_v}$ is the realization of $\phi_v$ in its Whittaker model. For finite $p$ the local integral $Z(s, W_{\phi_p}, \chi_p)$ equals $L(s, \pi_p \times \chi_p)$ whenever all the data is unramified. Moreover, if $\chi_p= \mathbf{1}$ is trivial,  $\pi_p$ is ramified, and $\phi_p$ is a taken to be paramodular newvector in $\pi_p$, then the local integral $Z(s, W_{\phi_p}, \mathbf{1})$ (once measures are normalized appropriately) equals $L(s, \pi_p )$. This shows that the global paramodular newform can serve as a test vector for the Novodvorsky integral representation for the spinor $L$-function in the untwisted case.

However, as explained earlier in the introduction, it is often useful to have a rich supply of test vectors for analytic applications (such as the subconvexity problem or non-vanishing of central $L$-values), because certain test vectors often work better than others for a specific application due to differing invariance properties or differing size of local integral.

We construct a new test vector $\phi_p$ for the Novodvorsky integral representation whenever $\pi_p$ is a simple supercuspidal and the conductor exponent of $\chi_p$ equals 0 or 1. The test vector $\phi_p$ is a particular diagonal translate of the minimal vector and we compute the local integral  $Z(s, W_{\phi_p}, \chi_p)$ corresponding to this test vector  in Proposition \ref{zeta-int-value-prop}.

As a consequence of this local computation, we explicitly write down a global integral representation for the spin $L$-function $L(s, \pi \times \chi)$ for a globally generic, cuspidal automorphic representation $\pi$ of $\GSp_4(\A)$ with trivial central character with each ramified local component of simple supercuspidal type, and an even Dirichlet character $\chi$ of squarefree conductor dividing that of $\pi$. We give two versions of the global formula, one where $\phi$ is Whittaker-normalized and the other where $\phi$ is $L^2$-normalized; the former is more suited for arithmetic applications and the latter more for analytic applications. For the exact statements, we refer the reader to Theorem \ref{global-novo-thm} and Corollary \ref{c:globalnov}.

\subsection{Generalized B\"ocherer's conjecture and the refined Gan--Gross--Pra\-sad period formula for \texorpdfstring{$(\SO_5, \SO_2)$}{}}
Let $f$ be a Siegel cusp form of degree 2 and weight $k$ for the group $\Sp_4(\Z)$. Assume that $f$ is a Hecke eigenform and let $d < 0$ be a fundamental discriminant.  B\"ocherer \cite{boch-conj} made a remarkable conjecture that relates the central $L$-value $L(1/2, f \times \chi_d)$ to the square of the sum of Fourier coefficients of $f$ corresponding to equivalence classes of forms of discriminant~$d$.
In a previous work with Dickson \cite{DPSS15}, we formulated an explicit generalization of B\"ocherer's conjecture by interpreting it as a special case of the refined  Gan--Gross--Prasad (GGP) period conjecture for $(\SO_5, \SO_2)$ as stated by Liu \cite{yifengliu}. The refined GGP conjecture in this special case takes the form
\begin{equation}\label{e:refggp}
  \frac{|B(\phi, \Lambda)|^2}{\langle \phi, \phi \rangle} = C \frac{L(1/2, \pi \times \AI(\Lambda^{-1}))}{L(1, \pi, \Ad)L(1, \chi_{d})} \prod_v B_{\Lambda_v, \theta_v}(\phi_v),
\end{equation}
where $\phi$, corresponding to $\otimes_v \phi_v$, is an automorphic form inside a cuspidal automorphic representation $\pi$ of $\GSp_4(\A)$, $\Lambda$ is a character of $K^\times \bs \A_K^\times$ satistfying $\Lambda|_{\A^\times} = 1$, $B(\phi, \Lambda)$ is the global Bessel period, and $C$ is a constant.

 The corresponding local Bessel integrals $B_{\Lambda_v, \theta_v}(\phi_v)$ were explicitly computed in \cite{DPSS15} in some special cases when $\Lambda_p$ is trivial and $\pi_p$ is Iwahori-spherical. This gives an explicit generalization of B\"ocherer's conjecture \cite[Thm.~1.13]{DPSS15} for certain Siegel cusp forms of squarefree level. In a couple of groundbreaking recent works, Furusawa and Morimoto \cite{FM21, FM22} have  now proved the above refined  GGP period conjecture \eqref{e:refggp}. An immediate corollary of their work is the proof of the explicit generalized Böcherer’s conjecture in the square-free level case that was formulated by us in \cite{DPSS15}. 
 
 The local Bessel integrals were also estimated in some special cases when $\Lambda_p$ is non-trivial and $\pi_p$ is unramified in \cite{CMS23} where these bounds were combined with the refined  GGP period conjecture to  bound the size of Fourer coefficients and the sup-norms of  Siegel cusp forms.

However, the local $p$-adic integrals $B_{\Lambda_v, \theta_v}(\phi_v)$ appearing in \eqref{e:refggp} have so far not been computed in \emph{any} case when  $\pi_p$ is not Iwahori-spherical. Consequently no explicit generalization of B\"ocherer’s conjecture for non square-free levels has been stated or proved. In this work, we fill this gap when $\pi_p$ is a simple supercuspidal representation.   We prove that a suitable diagonal translate of the minimal vector can be taken as a test vector $\phi_p$ for the local Bessel integral, and we compute the corresponding local integral $B_{\Lambda_p, \theta_p}(\phi_p)$ in Proposition~\ref{Bessel-model-prop}.  Our new test vector works for any character $\Lambda_p$ as long as its conductor exponent is not too small (in contrast to the result of \cite{DPSS15} which only applied to trivial $\Lambda_p$). The explicit refined global GGP period conjecture for this choice of test vector is proved in Theorem \ref{Bessel-per-thm} by combining our calculations with the recent work of Furusawa and Morimoto \cite{FM22}. Theorem \ref{Bessel-per-thm} can be reformulated in the classical language to give a proof of an explicit generalization of Böcherer’s conjecture for a certain class of forms of non square-free level, but we do not carry this out here in the interest of brevity.

\subsection{A special case of the Lapid--Mao formula}

Given a generic cuspidal automorphic form $\phi$ on $G(\A)$ for a quasi-split reductive group $G$, there are two natural ways  to specify a normalization of $\phi$. One is to set some particular Fourier--Whittaker coefficient equal to 1. For example, in the classical theory of Hecke eigenforms $f$ for $\SL_2(\Z)$, it is  natural, especially for arithmetic applications, to normalize $f$ by setting $a_f(1)=1$. This normalization ensures that all Fourier coefficients are algebraic integers. The other way to normalize $f$ is to set the Petersson norm $\langle f, f \rangle$ equal to 1, which is often useful for analytic applications.  The relation between these two normalizations is expressed by the well-known identity (see, e.g., \cite[(5.101)]{IKbook})
\begin{equation}\label{e:petgl2}
 \frac{|a_f(1)|^2}{\langle f, f \rangle} = \frac{2^k}{L(1, \pi_f, \Ad)},
\end{equation}
where $\pi_f$ is the automorphic representation generated by $f$ and $L(s, \pi_f, \Ad)$ denotes the (complete) adjoint $L$-function. The above identity is crucial for numerous applications in analytic number theory.

Lapid and Mao \cite{LM15} made a remarkable conjecture vastly generalizing \eqref{e:petgl2}. They proved the conjecture for cusp forms on $\GL_n$ using the theory of Rankin--Selberg integrals developed by
Jacquet, Piatetski-Shapiro, and Shalika. Moreover, in \cite{LM17}, they established an analogous formula in the metaplectic case.  In the special case $G=\GSp_4$, the Lapid--Mao conjecture was recently proved by Furusawa and Morimoto \cite[Theorem 6.3]{FM22} as part of their remarkable work on the refined GGP conjecture. Precisely, under the assumption that $\pi$ is a tempered automorphic representation of $\GSp_4(\A)$,  Furusawa and Morimoto proved that
\begin{equation}\label{e:lapidmaogsp4}
 \frac{|W_\phi(1) |^2   }{\langle \phi, \phi\rangle} =  2^{-c}\frac{\zeta^{S} (2) \zeta^{(S)}(4) }{L^{S}(1,  \pi, \Ad)}  \prod_{v \in S}J_0(\phi_v),
\end{equation}
where $\phi$, corresponding to $\otimes_v \phi_v$, is a cusp form in the space of $\pi$, the function $W_\phi$ is the Whittaker period  associated to $\phi$, the set $S$ consists of places such that all the local data is unramified outside~$S$, $c \in \{1,2\}$ is an integer depending on the Arthur packet of $\phi$, and $J_0(\phi_v)$ is a local Whittaker integral defined as the Whittaker coefficient of the matrix coefficient of~$\phi_v$.

However, for applications, it is often important to have an exact formula where the quantities $J_0(\phi_v)$ at the  bad places $v \in S$ are explicitly written down. So far, there has been little progress in this direction. Chen and Ichino \cite{chen-ichino} computed $J_0(\phi_v)$ at $v=\infty$ for $\pi_\infty$ a principal series or large discrete series representation and $\phi_\infty$ a vector of minimal weight. They also computed  $J_0(\phi_v)$ at a finite place $v=p$ for the $\pi_p$ with conductor exponent equal to 1, and $\phi_p$ a paramodular newvector. Apart from this, we are not aware of any other case where $J_0(\phi_v)$ has been computed.

In this work, we compute the local Whittaker integral $J_0(\phi_p)$ when $\pi_p$ is a simple supercuspidal representation and $\phi_p$ is either a diagonal translate of a minimal vector (Prop.~\ref{J0-non-vanishing}) or a paramodular newvector (Prop.~\ref{J0-prop}). These local results lead to an explicit identity between Petersson norms and Whittaker coefficients of cusp forms $\phi \in \pi$ with the above local constraints; we refer the reader to Theorem \ref{t:lapidmaoexplicit} for the exact statement of this result.

\subsection{Large values of paramodular newforms} Given  an $L^2$-normalised cuspidal automorphic form $\phi$ on some group $G$, it is of great interest to bound $\|\phi\|_\infty$ in terms of its defining parameters. This is a highly active area of research with connections to geometric analysis and mathematical physics and has seen an explosion of recent activity. Strong upper bounds for $\|\phi\|_\infty$  often imply strong \emph{subconvexity bounds} for certain $L$-functions \cite{HS19} thus linking this problem to one of the most important problems in number theory. On the other hand, large lower bounds for $\|\phi\|_\infty$ give counterexamples to the random wave model from quantum mechanics, and even more remarkably, appear to have unexpected connections with functoriality. We refer the reader to the introductions of \cite{blomer-harcos-milicevic-maga, sup-norm-minimal-compact, HS19}  for brief discussions of some of these connections and various recent results.

Our focus here is on lower bounds for the sup-norm. Normalize the measure of the underlying space $Z_G(\A)G(\Q)\bs G(\A)$ so that it has volume equal to 2 and consider a family $\F$ of $L^2$-normalized  cuspidal automorphic forms on $G$. For each $\phi \in \F$, we have a ``trivial" lower bound $\|\phi\|_\infty \ge 1$ coming from the triangle inequality.  We say that the family $\F$ takes large values if something stronger is true, namely that for each $\phi \in \F$ we have $\|\phi\|_\infty \gg_\F C(\phi)^\delta$  for some fixed $\delta>0$, where $C(\phi)$ denotes the analytic conductor of the automorphic representation attached to $\phi$.

In the literature, one finds two main sources of large values. First, large values can  arise from the unusual behaviour of certain \emph{global} lifts. For example, Rudnick and Sarnak \cite{MR1266075} discovered that among Maass forms on the group $\SO(3,1)$ with eigenvalue $\lambda$, the ones that are theta lifts from $\SL_2$ have $L^\infty$-norm of the order of magnitude at least  $\lambda^{1/2}$, in contrast to a ``typical" Maass form whose $L^\infty$-norm is expected to be at most $\lambda^\epsilon$ according to the random wave model. This has been generalized to many higher rank cases by Brumley and Marshall \cite{brumley-marshall}. In the above results, the large values are obtained in the bulk, i.e., in a fixed compact set.
One can also have a very different source for large values coming from the shape of \emph{local} Whittaker functions. This phenomenon has been explored by  Templier \cite{templier-large}, the second-named author \cite{sahasupwhittaker, HNS}, Assing \cite{assing18} and Brumley--Templier \cite{brumley-templier}. In all these cases, the large values are obtained near the cusp.

In this work, we demonstrate a new phenomenon. We show that certain paramodular newforms take large values in the conductor aspect in a fixed compact set, \emph{despite} their source being the behaviour of the local Whittaker integrals associated to paramodular newvectors. More precisely, we prove the following theorem.

\begin{theorem}[Theorem \ref{t:largevalues}]\label{t:largeintro}
Let $\pi\simeq\otimes_v \pi_v$ be an irreducible, unitary, cuspidal, globally generic automorphic representation of $\GSp_4(\A)$ with trivial central character such that $\pi_p$ is a simple supercuspidal representation for each prime $p$ where it is ramified. Let $N = \prod_{p|N} p^5$ be the conductor of $\pi$, and assume that $\pi_\infty$ is a  discrete series representation.
$K(p^{a(\pi_p)})$ at each prime $p$.
Let $\phi$ be an automorphic form in the space of $\pi$. Suppose that $\phi$ is a newform with respect to the paramodular subgroup of level $N$ and $\phi_\infty$ is a lowest weight vector in~$\pi_\infty$.
Then we have $$\sup_{g \in U(\Q) \bs U(\A)}\frac{|\phi(g)|}{\langle \phi, \phi \rangle^{1/2}}  \gg_{\pi_\infty, \eps} N^{1/2 - \eps} $$ where $U$ is the unipotent radical of the Borel parabolic (so that $U(\Q) \bs U(\A)$ is compact).
\end{theorem}

We note that trace formula arguments imply that infinitely many paramodular newforms of the type considered in Theorem \ref{t:largeintro} exist   (see, e.g., \cite[Thm.~1.2]{kim-shin-temp}).   In Theorem \ref{t:largevalues} we prove a somewhat more general result where we give an explicit compact subset of $U(\R)$ where the large values are attained, include the case where $\pi_p$ has conductor exponent 1~or~5, and allow $\pi_\infty$ to be a principal series representation. We suspect that the large value phenomenon holds for any family of paramodular newforms of $\GSp_4$ whose conductors tend to infinity.

\subsection{Structure of the paper}In Sect. \ref{simplescsec}, we introduce the local notations and develop the theory of simple supercuspidal representations. In Sect. \ref{s:localprelims} we compute various local matrix coefficients and local integrals. We apply these results in Sect. \ref{s:global} to prove our main global results.

\subsection{Acknowledgements}A.S. acknowledges the support of the Engineering and Physical Sciences Research Council (grant number EP/T028343/1). We thank the anonymous referee for a careful reading of the manuscript and for many helpful comments, which have improved this paper.
\section{Simple supercuspidals}\label{simplescsec}
In this section we develop the theory of simple supercuspidal representations of $\GSp_4$, following the theory of simple supercuspidals of $\GL_n$ outlined in \cite{KnightlyLi2015}. In Sect.~\ref{paravecsec} we show that simple supercuspidals with trivial central character have conductor exponent $5$, and exhibit the local paramodular newvector in the standard model of these representations. In Sect.~\ref{primsec} we express the newvector in terms of another distinguished element called the \emph{minimal vector}.
\subsection{Basic local notations}\label{basiclocalnotsec}
Throughout Sects.~\ref{simplescsec} and \ref{s:localprelims}, $F$ will be a non-archimedean local field of characteristic zero. Let $\OF$ be the ring of integers of $F$, with maximal ideal $\p$ and uniformizer $\varpi$. Let $\mathbf{k}=\OF/\p$ be the residue class field, and $q$ its cardinality. For $x \in F$, let $| x |$ denote the normalized absolute value of $x$, so that $|\varpi| = q^{-1}$.  Let $\psi$ be a character of $F$ which is trivial on $\OF$ but non-trivial on $\p^{-1}$. Let $\psi_0$ be the character of $F$ defined by $\psi_0(x) = \psi(\varpi^{-1} x)$, so that $\psi_0$ is non-trivial on $\OF$ but trivial on~$\p$. Hence $\psi_0$ induces a non-trivial character of $\mathbf{k}$.

We use the Haar measure $dx$ on $F$ that assigns $\OF$ volume 1, and we use the Haar measure $d^\times x$ on $F^\times$ that assigns $\OF^\times$ volume 1. So we have $d^\times x= (1-q^{-1})^{-1} \frac{ dx}{|x|}.$

Let
\begin{equation}\label{gsp4defeq}
 G=\GSp_4(F)=\{g\in\GL_4(F):\ ^tgJg=\mu(g)J\text{ for some }\mu(g)\in F^\times\},\quad J=\begin{bsmallmatrix}&&&1\\&&1\\&-1\\-1\end{bsmallmatrix}.
\end{equation}
The elements
\begin{equation}\label{s1s2defeq}
 s_1=\begin{bsmallmatrix}&1\\1\\&&&1\\&&1\end{bsmallmatrix},\qquad s_2=\begin{bsmallmatrix}1\\&&1\\&-1\\&&&1\end{bsmallmatrix}
\end{equation}
represent generators for the $8$-element Weyl group of $G$. Let $K=\GSp_4(\OF)$ be the standard hyperspecial maximal compact subgroup of $G$.

Let $N'$ be the normalizer of the diagonal subgroup $M$ of $G$. Then
\begin{equation}\label{Aetatheoremeq1}
 N'=\bigsqcup_{w\in W}wM,
\end{equation}
where $W=\{1,s_1,s_2,s_1s_2,s_2s_1,s_1s_2s_1,s_2s_1s_2,s_1s_2s_1s_2\}$ represents the 8 elements of the Weyl group. The affine Bruhat decomposition implies that
\begin{equation}\label{Aetatheoremeq2}
 G=K'N'K',
\end{equation}
where
\begin{equation}\label{Kpdefeq}
 K':=G\cap\begin{bsmallmatrix}1+\p&\OF&\OF&\OF\\\p&1+\p&\OF&\OF\\\p&\p&1+\p&\OF\\\p&\p&\p&1+\p\end{bsmallmatrix}.
\end{equation}

For a representation $\pi$ of $G$ on a space $V$, when the representation $\pi$ is clear from the context, we will often use the shorthand $gv$ or $g\cdot v$ to denote $\pi(g)v$.  We  use $V_\pi$ to denote the space of~$\pi$.

For $\alpha, \beta \in F^\times$,  define the element $d_{\alpha, \beta}  \in M$ by
\begin{equation}\label{dalhabeta}
 d_{\alpha, \beta} := {\rm diag}(\alpha^2\beta, \alpha \beta, \alpha, 1).
\end{equation}

For each non-negative integer $n$ we define the Klingen congruence subgroup $\Kl{n}$ of level $n$ by
\begin{equation}\label{klingendefeq}
 \Kl{n} =\{k \in G\cap\begin{bsmallmatrix} \OF&\OF&\OF&\OF\\ \p^n &\OF&\OF&\OF\\ \p^n &\OF&\OF&\OF \\ \p^n &\p^n&\p^n&\OF \end{bsmallmatrix}:\, \det(k) \in \OF^\times\}
\end{equation}
and the paramodular group $\K{n}$ of level $n$ by
\begin{equation}\label{paradefeq}
 \K{n} =\left\{k \in G\cap\begin{bsmallmatrix} \OF&\OF&\OF&\p^{-n}\\ \p^n &\OF&\OF&\OF\\ \p^n &\OF&\OF&\OF \\ \p^n &\p^n&\p^n&\OF \end{bsmallmatrix}:\, \det(k) \in \OF^\times \right\}.
\end{equation}

\subsection{Affine generic characters}\label{affgencharsec}
Let $Z\cong F^\times$ be the center of $G$ and let $H=ZK'$. We fix a character $\omega$ of $Z$, trivial on $1+\p$. For $t_1,t_2,t_3\in\OF^\times$, we define a character $\chi:H\to\C^\times$ by
\begin{equation}\label{chidefeq}
 \chi(z\begin{bsmallmatrix} *&r_1&*&*\\ *&*&r_2&*\\ *&*&*&*\\\varpi r_3&*&*&*\end{bsmallmatrix})=\omega(z)\psi_0(t_1r_1+t_2r_2+t_3r_3).
\end{equation}
Such $\chi$ are called \emph{affine generic characters}. For fixed $\omega$, there are $(q-1)^3$ affine generic characters, corresponding to the choices of $t_1,t_2,t_3$ modulo $1+\p$. We sometimes write $\chi_{t_1,t_2,t_3}$ instead of $\chi$.

The group $M\cap K$ normalizes $K'$, and hence acts on the set of affine generic characters. For $m\in M\cap K$, let $\chi^m(x)=\chi(mxm^{-1})$. If $m={\rm diag}(a,b,cb^{-1},ca^{-1})$, then
\begin{equation}\label{chimactioneq}
 (\chi_{t_1,t_2,t_3})^m=\chi_{t_1ab^{-1},t_2b^2c^{-1},t_3a^{-2}c}.
\end{equation}
By choosing $m={\rm diag}(1,t_1,t_1t_2,t_1^2t_2)$, we see that the orbit of $\chi$ contains a character of the form $\chi_{1,1,t}$, and in fact a unique such character. Hence there are exactly $q-1$ orbits of affine generic characters with a fixed $\omega$.
\begin{lemma}\label{chimactionlemma}
 Let $t_1,t_2,t_3,\ell_1,\ell_2,\ell_3$ be elements of $\OF^\times$. Then $\chi_{t_1,t_2,t_3}$ and $\chi_{\ell_1,\ell_2,\ell_3}$ lie in the same $M\cap K$-orbit if and only if $t_1^2t_2t_3=\ell_1^2\ell_2\ell_3$ as elements of $\mathbf{k}^\times$.
\end{lemma}
\begin{proof}
Easy to see from \eqref{chimactioneq}.
\end{proof}

For $\chi=\chi_{t_1,t_2,t_3}$, let
\begin{equation}\label{gchidefeq}
 g_\chi=\begin{bsmallmatrix}0&0&1&0\\0&0&0&-1\\-\varpi t_2/t_3&0&0&0\\0&\varpi t_2/t_3&0&0\end{bsmallmatrix}.
\end{equation}
Note that $g_\chi$ equals ${\rm diag}(1,1,-t_2/t_3,-t_2/t_3)$ times the Atkin-Lehner element $u_1$, where
\begin{equation}\label{undefeq}
 u_n=\begin{bsmallmatrix}&&1\\&&&-1\\\varpi^n\\&-\varpi^n\end{bsmallmatrix}
\end{equation}
is the usual Atkin-Lehner element of level $n$. An easy calculation confirms that $g_\chi$ normalizes $K'$ and $H$, and that
\begin{equation}\label{chigchieq}
 \chi(g_\chi^{-1}hg_\chi)=\chi(h)\qquad\text{for all }h\in H.
\end{equation}
Observe that $g_\chi^2=-\varpi\frac{t_2}{t_3}I_4$.
\subsection{The induced representation}
Given an affine generic character $\chi$, define
\begin{equation}\label{pichidefeq}
 \pi_\chi=\cInd^G_H(\chi).
\end{equation}
The standard model $A_\chi$ of $\pi_\chi$ consists of smooth functions $f:G\to\C$ with the transformation property $f(hg)=\chi(h)f(g)$ for $g\in G$ and $h\in H$, which are compactly supported modulo $Z$. Note that $\pi_\chi$ is a representation for which the center acts via the character $\omega$.
\begin{proposition}\label{supercuspidalprop}
 Any irreducible subrepresentation of $\pi_\chi$ is supercuspidal.
\end{proposition}
\begin{proof}
See the proof of Proposition 3.1 of \cite{KnightlyLi2015}.
\end{proof}

For an affine generic character $\eta$, let
\begin{equation}\label{Aetadefeq}
 A^\eta=\{f\in A_\chi:\,\pi_\chi(h)f=\eta(h)f\text{ for all }h\in H\}.
\end{equation}
A non-zero element of $A^\chi$ is given by
\begin{equation}\label{f0defeq}
 f_0(h)=\begin{cases}
         \chi(h)&\text{if }h\in H,\\
         0&\text{otherwise.}
        \end{cases}
\end{equation}
\begin{proposition}\label{etachiprop}
 Let $\chi$ and $\eta$ be affine generic characters. Suppose $\phi\in A^\eta$. If $\phi(x)\neq0$, then
 \begin{equation}\label{etachipropeq1}
  \eta(h)=\chi(xhx^{-1})\qquad\text{for all }h\in H\cap x^{-1}Hx.
 \end{equation}
 This condition is independent of the choice of representative $x$ for the double coset $HxH$. Conversely, if $x\in G$ is any element satisfying \eqref{etachipropeq1}, then there exists a unique element $\phi_x\in A^\eta$ supported on $HxH$ and satisfying $\phi_x(x)=1$.

 An element $x\in G$ satisfies \eqref{etachipropeq1} if and only if $g_\chi x$ satisfies \eqref{etachipropeq1}. For such $x$, the set
 \begin{equation}\label{etachipropeq2}
  \{\phi_x,\phi_{g_\chi x}\}\subset A^\eta
 \end{equation}
 is linearly independent.
\end{proposition}
\begin{proof}
See the proof of Proposition 3.3 of \cite{KnightlyLi2015}.
\end{proof}

\begin{theorem}\label{Aetatheorem}
 Let $t_1,t_2,t_3,\ell_1,\ell_2,\ell_3$ be elements of $\OF^\times$. Let $\chi=\chi_{t_1,t_2,t_3}$ and $\eta=\chi_{\ell_1,\ell_2,\ell_3}$. Let $A=A_\chi$ be the standard model of $\pi_\chi$. Then the following are equivalent.
 \begin{enumerate}
  \item $A^\eta\neq0$
  \item $t_1^2t_2t_3=\ell_1^2\ell_2\ell_3$ as elements of $\mathbf{k}^\times$.
  \item $\chi$ and $\eta$ are in the same $M\cap K$-orbit.
 \end{enumerate}
 If these conditions hold, then $\chi^{m_0}=\eta$, where
 \begin{equation}\label{Aetatheoremeq15}
  m_0={\rm diag}\Big(1,\frac{t_1}{\ell_1},\frac{t_1t_2}{\ell_1\ell_2},\frac{t_1^2t_2}{\ell_1^2\ell_2}\Big).
 \end{equation}
 Furthermore, the elements $\phi_{m_0}$ and $\phi_{g_\chi m_0}$ in \eqref{etachipropeq2} form a basis of $A^\eta$.
\end{theorem}
\begin{proof}
The proof is similar to that of \cite[Thm.~3.4]{KnightlyLi2015}.
\end{proof}
\subsection{Definition of simple supercuspidal representations of \texorpdfstring{$\GSp_4$}{}}\label{inddecompeq}
Let
\begin{equation}\label{Edefeq}
 E=\{g\in G:\,v(\mu(g))\in2\Z\},
\end{equation}
where $v$ is the normalized valuation on $F$. Then $E$ is a subgroup of $G$ of index $2$. We have $G=E\sqcup g_\chi E$.

Let $t_1,t_2,t_3\in\OF^\times$ and $\chi=\chi_{t_1,t_2,t_3}$ be the associated affine generic character. As before, let $A$ be the standard model of the induced representation $\pi_\chi$. Let $A_0$ be the subspace of $A$ consisting of functions whose support is in $E$, and let $A_1$ be the subspace of $A$ consisting of functions whose support is in $g_\chi E$. Then
\begin{equation}\label{A01decompeq}
 A=A_0\oplus A_1
\end{equation}
by the argument in \cite[Sect.~4.1]{KnightlyLi2015}. Evidently, $A_0$ and $A_1$ are $E$-submodules of $A$, and $A_0$ can be identified with the space of the compactly induced representation
\begin{equation}\label{sigmachidefeq}
 \sigma_\chi=\cInd^E_H(\chi).
\end{equation}
The proof of the following result is similar to that of \cite[Prop.~4.1]{KnightlyLi2015}.
\begin{proposition}\label{A0irredprop}
 The representation $(\sigma_\chi,A_0)$ of $E$ is irreducible. Two such representations $\sigma_\chi$ and $\sigma_\eta$ are equivalent if and only if $\eta=\chi^m$ for some $m\in M\cap K$.
\end{proposition}
Define an operator $L$ on $A$ by
\begin{equation}\label{Ldefeq}
 (L\phi)(x)=\phi(g_\chi^{-1}x),\qquad x\in G.
\end{equation}
It is clear that $L$ induces $E$-isomorphisms $A_0\to A_1$ and $A_1\to A_0$. Hence the $E$-module $A_1$ is also irreducible and isomorphic to $\sigma_\chi$. As a consequence, we obtain the following result, which is proven just like \cite[Cor.~4.3]{KnightlyLi2015}.
\begin{proposition}\label{chietaequivprop}
 Given two affine generic characters $\chi$ and $\eta$, the induced representations $\pi_\chi$ and $\pi_\eta$ of $G$ are equivalent if and only if $\chi$ and $\eta$ belong to the same $M\cap K$-orbit.
\end{proposition}

We now decompose $\pi_\chi$ into irreducibles. By Lemma~\ref{chimactionlemma} and Proposition~\ref{chietaequivprop}, we may assume that $\chi=\chi_{1,1,t}$ for $t\in\OF^\times$. In this case $g_\chi^2=-\frac{\varpi}tI_4$. Let $\zeta\in\C$ satisfy $\zeta^2=\omega(-\varpi/t)$. It is straightforward to verify that
\begin{equation}\label{zetaL2eq}
 (\zeta L)^2\phi=\phi\qquad\text{for all }\phi\in A.
\end{equation}
Define
\begin{equation}\label{Sigmazetadefeq}
 \Sigma_\zeta=\{\phi+\zeta L\phi\::\:\phi\in A_0\}.
\end{equation}
The map $A_0\to\Sigma_\zeta$ given by $\phi\mapsto\phi+\zeta L\phi$ is an isomorphism of $E$-modules. Thus $\Sigma_\zeta$ is an irreducible $E$-module isomorphic to $\sigma_\chi$. For $\phi\in A_0$, set $\psi=\pi_\chi(g_\chi)\phi\in A_1$. Then
\begin{equation}
 \pi_\chi(g_\chi)(\phi+\zeta L\phi)=\psi+\zeta L\psi=\zeta L\psi+(\zeta L)^2\psi=\xi+\zeta L\xi,
\end{equation}
where $\xi=\zeta L\psi\in A_0$. This shows that $\Sigma_\zeta$ is $\pi(g_\chi)$-invariant, and hence is a $G$-submodule of $A$. It is an irreducible $G$-module, since it is irreducible as an $E$-module. We denote the action of $G$ on $\Sigma_\zeta$ by $\sigma_\chi^\zeta$. By Proposition~\ref{supercuspidalprop}, $\sigma_\chi^\zeta$ is a supercuspidal representation of $G$. We call it a \emph{simple supercuspidal representation}.
\begin{theorem}\label{pichizetatheorem}
 Let $\chi=\chi_{1,1,t}$ with $t\in\OF^\times$, and fix $\zeta\in\C^\times$ with $\zeta^2=\omega(-\varpi/t)$. Then
 \begin{equation}\label{pichizetatheoremeq1}
  \pi_\chi=\sigma_\chi^\zeta\oplus\sigma_\chi^{-\zeta},
 \end{equation}
 and the two supercuspidal representations $\sigma_\chi^\zeta$ and $\sigma_\chi^{-\zeta}$ are not isomorphic.
\end{theorem}
\begin{proof}
It is easy to see that $A=\Sigma_\zeta\oplus\Sigma_{-\zeta}$ as vector spaces, proving \eqref{pichizetatheoremeq1}. To prove the last assertion, observe
\begin{equation}\label{pichizetatheoremeq2}
 \Hom_G(\pi_\chi,\pi_\chi)\cong\Hom_H(\chi,\pi_\chi)\cong A^\chi,
\end{equation}
and $\dim A_\chi=2$ by Theorem \ref{Aetatheorem}.
\end{proof}

Since $A_0\cong\Sigma_\zeta$ as vector spaces, the representation $\sigma_\chi^\zeta$ has a model on $A_0$. It is given by
\begin{equation}\label{A0modeleq}
 (\sigma_\chi^\zeta(g)\phi)(x)=\begin{cases}
                                \phi(xg)&\text{if }g\in E,\\
                                \zeta\phi(g_\chi^{-1}xg)&\text{if }g\in g_\chi E.
                               \end{cases}
\end{equation}
We consider in particular the case that the central character $\omega$ is trivial. Then $\zeta^2=1$. Instead of \eqref{Sigmazetadefeq} we will write
\begin{equation}\label{Sigmazetadefe2}
 \Sigma_\pm=\{\phi\pm L\phi\::\:\phi\in A_0\},
\end{equation}
and denote the action of $G$ on this space by $\sigma_\chi^\pm$. We have $\pi_\chi=\sigma_\chi^+\oplus\sigma_\chi^-$. Consider the group
\begin{equation}\label{Hprimedefeq}
 H'=H\sqcup g_\chi H.
\end{equation}
The character $\chi=\chi_{1,1,t}$ of $H$ admits two different extensions $\chi^+$ and $\chi^-$ to a character of $H'$, given by
\begin{equation}\label{chipmdefeq}
 \chi^\pm(h)=\chi(h) \text{ and } \chi^\pm(g_\chi h)=\pm\chi(h)
\end{equation}
for $h\in H$. Since $\ind_H^{H'}(\chi)=\chi^+\oplus\chi^-$, we see that
\begin{equation}\label{chipmindeq}
 \sigma_\chi^\pm=\cInd^G_{H'}(\chi^\pm).
\end{equation}
\subsection{Paramodular vectors}\label{paravecsec}
Let $\chi=\chi_{1,1,t}$ with $t\in\OF^\times$. Let $A$ be the standard space of the representation $\pi_\chi$ defined in~\eqref{pichidefeq}. In this section we consider paramodular vectors in $A$, and hence assume that the central character of $\pi_\chi$ is trivial. In this case the number $\zeta$ appearing in \eqref{zetaL2eq} is $\pm1$. Hence $\Sigma_+$ and $\Sigma_-$ as in \eqref{Sigmazetadefe2} are the two irreducible constituents of $\pi_\chi$. Let $\sigma_\chi^\pm$ be the representation of $G$ on $\Sigma_\pm$.

For a non-negative integer $n$ recall the paramodular group $\K{n}$ defined in \eqref{paradefeq}. Let $A(n)$ be the subspace of $A$ consisting of $\K{n}$-invariant vectors.

\begin{lemma}\label{HgKwelldeflemma}
 For $g\in G$ and a non-negative integer $n$ the following are equivalent.
 \begin{enumerate}
  \item There exists $f\in A(n)$ with $f(g)\neq0$.
  \item $\chi$ is trivial on $H\cap g\K{n}g^{-1}$.
 \end{enumerate}
\end{lemma}
\begin{proof}
Straightforward verification.
\end{proof}

We consider double cosets of the form $Hg\K{n}$, where $g\in M$. By adjusting by units and an element of the center, we see that every such double coset is of the form
\begin{equation}\label{HMKeq1}
 H d_{\varpi^i,\varpi^j}\K{n},\qquad i,j\in\Z.
\end{equation}
It is an easy exercise to show that the pair $(i,j)$ is uniquely determined by the double coset. The following result is an easy consequence of Lemma~\ref{HgKwelldeflemma}.
\begin{lemma}\label{HMKlemma}
 The double coset \eqref{HMKeq1} supports an element of $A(n)$ if and only if $i,j\geq1$ and $2i+j\leq n-2$.
\end{lemma}
Let $A^*(n)$ be the space of $f\in A(n)$ that are supported on double cosets of the form \eqref{HMKeq1}. Lemma~\ref{HMKlemma} shows that
\begin{equation}\label{dimAstarneq}
 \dim A^*(n)=
 \begin{cases}
  \Big\lfloor\frac{(n-3)^2}4\Big\rfloor&\text{for }n\geq5,\\[1ex]
  0&\text{for }n\leq4.
 \end{cases}
\end{equation}
Recall the decomposition \eqref{A01decompeq}, and define $A_0(n)=A_0\cap A(n)$ and $A_1(n)=A_1\cap A(n)$. Assuming that $(i,j)$ satisfies $i,j\geq1$ and $2i+j\leq n-2$, let $f_{i,j}^{(n)}\in A(n)$ be the vector supported on the double coset \eqref{HMKeq1} and taking the value $1$ on the diagonal representative. We have $f_{i,j}^{(n)}\in A_0$ if $j$ is even and $f_{i,j}^{(n)}\in A_1$ if $j$ is odd. For example, $f_{1,1}^{(5)}$ is in $A_1(5)$. Explicitly,
\begin{equation}\label{f115eq}
 f_{1,1}^{(5)}(h d_{\varpi, \varpi}k)=\chi(h)\qquad\text{for }h\in H,\;k\in\K{5}.
\end{equation}
 We have
\begin{equation}\label{gchiuneq}
 g_\chi=\begin{bsmallmatrix}0&0&1&0\\0&0&0&-1\\-\varpi/t&0&0&0\\0&\varpi/t&0&0\end{bsmallmatrix}=\begin{bsmallmatrix}1\\&1\\&&-\varpi^{1-n}t^{-1}\\&&&-\varpi^{1-n}t^{-1}\end{bsmallmatrix}u_n,
\end{equation}
with $u_n$ defined in (\ref{undefeq}). Now let $f_1:=f_{1,1}^{(5)}$ and set $f_0:=Lf_1$, where $L$ is the operator defined in \eqref{Ldefeq}. By \eqref{zetaL2eq} we have $L^2=1$, so that $f_1=Lf_0$. Let
\begin{equation}\label{fpmdefeq}
 f_\pm:=f_0\pm f_1.
\end{equation}
Since $f_1=Lf_0$, we see $f_\pm\in\Sigma_\pm$. A straightforward calculation using \eqref{gchiuneq} confirms that
\begin{equation}\label{fpmu5eq}
 \pi_\chi(u_5)f_\pm=\pm f_\pm.
\end{equation}
For the following lemma, let $T_{0,1}$ be the paramodular Hecke operator of level $5$ defined in (6.3) of \cite{NF}.
\begin{lemma}\label{sigmapmlevellemma}
 We have
 $$
  (T_{0,1}f_\pm)( d_{\varpi, \varpi})=0.
 $$
\end{lemma}
\begin{proof}
By \cite[Lemma~6.1.2 i)]{NF},
$$
 (T_{0,1}f_\pm)( d_{\varpi, \varpi})=A+B+C+D,
$$
with
\begin{align*}
 A&=\sum_{x,y,z\in\OF/\p}f_\pm( d_{\varpi, \varpi}\begin{bsmallmatrix}1&&y&z\varpi^{-5}\\&1&x&y\\&&1\\&&&1\end{bsmallmatrix} d_{1,\varpi}),\\
 B&=\sum_{x,z\in\OF/\p}f_\pm( d_{\varpi, \varpi}\begin{bsmallmatrix}1&x&&z\varpi^{-5}\\&1\\&&1&-x\\&&&1\end{bsmallmatrix} d_{\varpi,\varpi^{-1}}),\\
 C&=\sum_{x,y\in\OF/\p}f_\pm( d_{\varpi, \varpi}t_5\begin{bsmallmatrix}1&&y\\&1&x&y\\&&1\\&&&1\end{bsmallmatrix} d_{1,\varpi}),\\
 D&=\sum_{x\in\OF/\p}f_\pm( d_{\varpi, \varpi}t_5\begin{bsmallmatrix}1&x\\&1\\&&1&-x\\&&&1\end{bsmallmatrix} d_{\varpi,\varpi^{-1}}),
\end{align*}
where
$$
 t_5=\begin{bsmallmatrix}1&&&-\varpi^{-5}\\&1\\&&1\\\varpi^5\end{bsmallmatrix}.
$$
We have
\begin{align*}
 A&=q^2\sum_{z\in\OF/\p}f_\pm( d_{\varpi, \varpi}\begin{bsmallmatrix}1&&&z\varpi^{-5}\\&1\\&&1\\&&&1\end{bsmallmatrix} d_{1,\varpi})\\
 &=q^2\sum_{z\in(\OF/\p)^\times}f_\pm( d_{\varpi, \varpi}\begin{bsmallmatrix}1&&&z\varpi^{-5}\\&1\\&&1\\&&&1\end{bsmallmatrix} d_{1,\varpi})\\
 &=q^2\sum_{z\in(\OF/\p)^\times}f_\pm( d_{\varpi, \varpi}\begin{bsmallmatrix}1\\&1\\&&1\\z^{-1}\varpi^5&&&1\end{bsmallmatrix}\begin{bsmallmatrix}&&&z\varpi^{-5}\\&1\\&&1\\-z^{-1}\varpi^5\end{bsmallmatrix}\begin{bsmallmatrix}1\\&1\\&&1\\z^{-1}\varpi^5&&&1\end{bsmallmatrix} d_{1,\varpi})\\
 &=q^2\sum_{z\in(\OF/\p)^\times}f_\pm( d_{\varpi, \varpi}\begin{bsmallmatrix}1\\&1\\&&1\\z^{-1}\varpi^5&&&1\end{bsmallmatrix}\begin{bsmallmatrix}&&&z\varpi^{-5}\\&1\\&&1\\-z^{-1}\varpi^5\end{bsmallmatrix} d_{1,\varpi})\\
 &=q^2\sum_{z\in(\OF/\p)^\times}f_\pm( d_{\varpi, \varpi}\begin{bsmallmatrix}1\\&1\\&&1\\z^{-1}\varpi^5&&&1\end{bsmallmatrix}\begin{bsmallmatrix}1\\&\varpi\\&&1\\&&&\varpi\end{bsmallmatrix}\begin{bsmallmatrix}&&&z\varpi^{-5}\\&1\\&&1\\-z^{-1}\varpi^5\end{bsmallmatrix})\\
 &=q^2\sum_{z\in(\OF/\p)^\times}f_\pm( d_{\varpi, \varpi}\begin{bsmallmatrix}1\\&1\\&&1\\z^{-1}\varpi^5&&&1\end{bsmallmatrix}\begin{bsmallmatrix}1\\&\varpi\\&&1\\&&&\varpi\end{bsmallmatrix})\\
 &=q^2\sum_{z\in(\OF/\p)^\times}f_\pm( d_{\varpi, \varpi}\begin{bsmallmatrix}1\\&\varpi\\&&1\\&&&\varpi\end{bsmallmatrix})\\
 &=0.
\end{align*}
Similar arguments show $B=C=D=0$.
\end{proof}

\begin{proposition}\label{sigmapmlevelprop}
 The representations $\sigma_\chi^\pm$ have conductor exponent $a(\sigma_\chi^\pm)=5$.
\end{proposition}
\begin{proof}
Let $\Sigma_\pm(n)$ be the space of $\K{n}$-invariant vectors in $\Sigma_\pm$. Above we produced a vector $f_\pm\in\Sigma_\pm(5)$. It follows that $\sigma_\chi^\pm$ is generic, and that $a(\sigma_\chi^\pm)\leq5$.

Using \cite[Theorem 8.4.7]{stable-jrs}, we know that generic supercuspidal representations have conductor exponent $\geq4$. So we only have to exclude the possibility that $a(\sigma_\chi^\pm)=4$.

Assume that $a(\sigma_\chi^\pm)=4$; we will obtain a contradiction. By our assumption, $\Sigma_\pm(4)$ is one-dimensional, spanned by a newvector $f^{\rm new}_\pm$. By \cite[Thm.~7.5.6]{NF} we know that $\Sigma_\pm(5)$ is $2$-dimensional, spanned by $\theta f^{\rm new}_\pm$ and $\theta' f^{\rm new}_\pm$; here $\theta$ and $\theta'$ are the level raising operators defined in Sect.~3.2 of \cite{NF}. The Atkin-Lehner eigenvectors in $\Sigma_\pm(5)$ are $(\theta+\theta')f^{\rm new}_\pm$ and $(\theta-\theta')f^{\rm new}_\pm$; one of them has $u_5$-eigenvalue $+1$ and the other has $u_5$-eigenvalue $-1$. It therefore follows from \eqref{fpmu5eq} that, at least up to multiples,
\begin{equation}\label{sigmapmlevelpropeq1}
 f_\pm=(\theta+\theta')f^{\rm new}_\pm\qquad\text{or}\qquad f_\pm=(\theta-\theta')f^{\rm new}_\pm.
\end{equation}
By \cite[Cor.~7.4.6]{NF}, $T_{0,1}f^{\rm new}_\pm=0$. (Here, $T_{0,1}$ is the Hecke operator at level $4$.) As a consequence, by \cite[Cor.~6.3.2]{NF}, $(\theta+\theta')f^{\rm new}_\pm$ and $(\theta-\theta')f^{\rm new}_\pm$ are eigenvectors for $T_{0,1}$ with eigenvalues $\pm q^2$. (Here, $T_{0,1}$ is the Hecke operator at level $5$.) Hence $f_\pm$ is an eigenvector for $T_{0,1}$ with eigenvalue $q^2$ or $-q^2$. However, this contradicts Lemma~\ref{sigmapmlevellemma}.
\end{proof}

By Theorem~7.5.6 of \cite{NF},
\begin{equation}\label{dimsigmapmneq}
 \dim\Sigma_\pm(n)=
 \begin{cases}
  \Big\lfloor\frac{(n-3)^2}4\Big\rfloor&\text{for }n\geq5,\\[1ex]
  0&\text{for }n\leq4.
 \end{cases}
\end{equation}
Comparison of \eqref{dimAstarneq} and \eqref{dimsigmapmneq} shows that the map
\begin{equation}\label{Astarisoeq}
 A^*(n)\longrightarrow\Sigma_\pm(n),\qquad f\longmapsto f\pm Lf,
\end{equation}
is an isomorphism. The newvector in $\sigma_\chi^\pm$ is given by
\begin{equation}\label{fpmeq}
 f^\pm_{\rm new}(g)=\begin{cases}
         \chi(h)&\text{if }g= h d_{\varpi, \varpi}k\text{ with }h\in H,\;k\in\K{5},\\
         \pm\chi(h)&\text{if }g= hg_\chi d_{\varpi, \varpi}k\text{ with }h\in H,\;k\in\K{5},\\
         0&\text{otherwise},
        \end{cases}
\end{equation}
where $g_\chi$ is as in \eqref{gchiuneq}.
Using the notations \eqref{Hprimedefeq} and \eqref{chipmdefeq}, this can also be written as
\begin{equation}\label{fpmeq2}
 f^\pm_{\rm new}(g)=\begin{cases}
         \chi^\pm(h)&\text{if }g= h d_{\varpi, \varpi}k\text{ with }h\in H',\;k\in\K{5},\\
         0&\text{otherwise}.
        \end{cases}
\end{equation}
\subsection{An expression for the newvector in terms of the minimal vector}\label{primsec}
Let $\chi=\chi_{1,1,t}$ with $t\in\OF^\times$, and assume that the central character $\omega$ of $\pi_\chi$ is trivial. As above, $\pi_\chi=\cInd^G_H(\chi)$ decomposes into two irreducible representations $\sigma_\chi^+$ and $\sigma_\chi^-$, with spaces $\Sigma_\pm$ as in \eqref{Sigmazetadefe2}.
Let $f^\pm_{\rm new}\in\Sigma_\pm$ be the local newvector of $\sigma_\chi^\pm$, given explicitly in \eqref{fpmeq} and \eqref{fpmeq2}. We define the \emph{minimal vector} $f^{\pm}_{\rm min}\in\Sigma_\pm$ by
\begin{align}\label{fmindefeq}
 f^\pm_{\rm min}(g)
      &=\begin{cases}
         \chi(h)&\text{if }g=h\text{ with }h\in H,\\
         \pm\chi(h)&\text{if }g=hg_\chi\text{ with }h\in H,\\
         0&\text{otherwise,}
        \end{cases}\nonumber\\
      &=\begin{cases}
         \chi^\pm(h)&\text{if }g=h\text{ with }h\in H',\\
         0&\text{otherwise}.
        \end{cases}
\end{align}
In Proposition~\ref{newminprop} below we will express $f^\pm_{\rm new}$ in terms of $f^\pm_{\rm min}$.

\begin{lemma}\label{Hcosetlemma1}
 With $\Gamma_1(\p)=\mat{1+\p}{\OF}{\p}{1+\p}$, we have
 \begin{align}\label{Hcosetlemma1eq1}
  \Gamma_1(\p)\mat{\varpi}{}{}{1}\GL(2,\OF)&=\bigsqcup_{u,v\in(\OF/\p)^\times}\bigsqcup_{x\in\OF/\p^2}\Gamma_1(\p)\mat{\varpi}{}{}{1}\mat{u}{}{}{v}\mat{1}{}{x}{1}\mat{}{1}{-1}{}\nonumber\\
   &\qquad\sqcup\bigsqcup_{u,v\in(\OF/\p)^\times}\bigsqcup_{x\in\OF/\p}\Gamma_1(\p)\mat{\varpi}{}{}{1}\mat{u}{}{}{v}\mat{1}{}{x\varpi}{1}.
 \end{align}
\end{lemma}
\begin{proof}
This is an exercise using
\begin{equation}\label{Hcoseteq1}
 \GL(2,\OF)=
 \bigsqcup_{x\in\OF/\p^2}
 \Gamma_0(\p^2)\mat{1}{}{x}{1}\mat{}{1}{-1}{}
 \sqcup
 \bigsqcup_{x\in\OF/\p}
 \Gamma_0(\p^2)\mat{1}{}{x\varpi}{1},
\end{equation}
which in turn follows from the easy-to-see decompositions $\Gamma_0(\p)=\bigsqcup_{x\in\OF/\p}\mat{1}{}{x\varpi}{1}\Gamma_0(\p^2)$ and $\GL(2,\OF)=\bigsqcup_{x\in\OF/\p}
 \Gamma_0(\p)\mat{1}{}{x}{1}\mat{}{1}{-1}{}
 \sqcup\Gamma_0(\p)$.
\end{proof}

\begin{lemma}\label{Hcosetlemma2}
 \begin{align}\label{Hcosetlemma2eq1}
  H d_{\varpi, \varpi}\Kl{5}&=\bigsqcup_{u,v\in(\OF/\p)^\times}\bigsqcup_{x\in\OF/\p^2}H d_{\varpi, \varpi}\begin{bsmallmatrix}uv\\&u\\&&v\\&&&1\end{bsmallmatrix}\begin{bsmallmatrix}1\\&1\\&x&1\\&&&1\end{bsmallmatrix}\begin{bsmallmatrix}1\\&&1\\&-1\\&&&1\end{bsmallmatrix}\nonumber\\
   &\qquad\sqcup\bigsqcup_{u,v\in(\OF/\p)^\times}\bigsqcup_{x\in\OF/\p}H d_{\varpi, \varpi}\begin{bsmallmatrix}uv\\&u\\&&v\\&&&1\end{bsmallmatrix}\begin{bsmallmatrix}1\\&1\\&x\varpi&1\\&&&1\end{bsmallmatrix}.
 \end{align}
\end{lemma}
\begin{proof}
Any element of $\Kl{5}$ can be written in the form
\begin{equation}\label{Hcosetlemma2eq2}
 \begin{bsmallmatrix}1&\OF&\OF&\OF\\&1&&\OF\\&&1&\OF\\&&&1\end{bsmallmatrix}\begin{bsmallmatrix}1\\\p^5&1\\\p^5&&1\\\p^5&\p^5&\p^5&1\end{bsmallmatrix}\begin{bsmallmatrix}\det(A)\\&A\\&&1\end{bsmallmatrix}\begin{bsmallmatrix}z\\&z\\&&z\\&&&z\end{bsmallmatrix}
\end{equation}
with $A\in\GL(2,\OF)$ and $z\in\OF^\times$. The upper and lower triangular part can be absorbed into~$H$ after commuting it past $d_{\varpi, \varpi}$. Note that the center is also absorbed, so that
\begin{equation}\label{Hcosetlemma2eq3}
 H d_{\varpi, \varpi}\Kl{5}=\bigcup_{A\in\GL(2,\OF)}H d_{\varpi, \varpi}\begin{bsmallmatrix}\det(A)\\&A\\&&1\end{bsmallmatrix}.
\end{equation}
By Lemma~\ref{Hcosetlemma1} we get a decomposition as in \eqref{Hcosetlemma2eq1}, even though not necessarily disjoint. However, the disjointness is then easy to check by direct computation.
\end{proof}

\begin{lemma}\label{Hcosetlemma3}
\begin{align}\label{Hcosetlemma3eq0}
  &H d_{\varpi, \varpi}\K{5}\nonumber\\
  &=\bigsqcup_{u,v\in(\OF/\p)^\times}\bigsqcup_{\substack{x\in\OF/\p^2\\y\in\OF/\p^2}}H d_{\varpi, \varpi}\begin{bsmallmatrix}uv\\&u\\&&v\\&&&1\end{bsmallmatrix}\begin{bsmallmatrix}1\\&1\\&x&1\\&&&1\end{bsmallmatrix}\begin{bsmallmatrix}1&&&y\varpi^{-5}\\&1\\&&1\\&&&1\end{bsmallmatrix}\begin{bsmallmatrix}1\\&&1\\&-1\\&&&1\end{bsmallmatrix}\nonumber\\
   &\qquad\sqcup\bigsqcup_{u,v\in(\OF/\p)^\times}\bigsqcup_{\substack{x\in\OF/\p\\y\in\OF/\p^2}}H d_{\varpi, \varpi}\begin{bsmallmatrix}uv\\&u\\&&v\\&&&1\end{bsmallmatrix}\begin{bsmallmatrix}1\\&1\\&x\varpi&1\\&&&1\end{bsmallmatrix}\begin{bsmallmatrix}1&&&y\varpi^{-5}\\&1\\&&1\\&&&1\end{bsmallmatrix}\nonumber\\
  &\qquad\sqcup\bigsqcup_{u,v\in(\OF/\p)^\times}\bigsqcup_{\substack{x\in\OF/\p^2\\z\in\OF/\p}}H d_{\varpi, \varpi}\begin{bsmallmatrix}uv\\&u\\&&v\\&&&1\end{bsmallmatrix}\begin{bsmallmatrix}1\\&1\\&x&1\\&&&1\end{bsmallmatrix}\begin{bsmallmatrix}1&&&z\varpi^{-4}\\&1\\&&1\\&&&1\end{bsmallmatrix}\begin{bsmallmatrix}&&&\varpi^{-5}\\&&1\\&-1\\-\varpi^5\end{bsmallmatrix}\nonumber\\
   &\qquad\sqcup\bigsqcup_{u,v\in(\OF/\p)^\times}\bigsqcup_{\substack{x\in\OF/\p\\z\in\OF/\p}}H d_{\varpi, \varpi}\begin{bsmallmatrix}uv\\&u\\&&v\\&&&1\end{bsmallmatrix}\begin{bsmallmatrix}1\\&1\\&x\varpi&1\\&&&1\end{bsmallmatrix}\begin{bsmallmatrix}1&&&z\varpi^{-4}\\&1\\&&1\\&&&1\end{bsmallmatrix}\begin{bsmallmatrix}&&&\varpi^{-5}\\&1\\&&1\\-\varpi^5\end{bsmallmatrix}.
\end{align}
The same decomposition holds with $H'$ instead of $H$. (See \eqref{Hprimedefeq} for the definition of $H'$.)
\end{lemma}
\begin{proof}
By considering multipliers, it is easy to see that if $Hg\K{5}=\sqcup Hr_i$ with some representatives $r_i$, then $H'g\K{5}=\sqcup H'r_i$. Hence the last assertion follows once we know~\eqref{Hcosetlemma3eq0}.

By Lemma 3.3.1 of \cite{NF} there is a disjoint decomposition
\begin{equation}\label{paramodularKlingendecompositioneq}
  \K{n}=\bigsqcup\limits_{y\in\OF/\p^n}
  \begin{bsmallmatrix}1&&&y\varpi^{-n}\\&1\\&&1\\&&&1\end{bsmallmatrix}\Kl{n}\sqcup\bigsqcup\limits_{z\in\OF/\p^{n-1}}
  t_n\begin{bsmallmatrix}1&&&z\varpi^{-n+1}\\&1\\&&1\\&&&1\end{bsmallmatrix}\Kl{n},
\end{equation}
where
\begin{equation}\label{tneq}
 t_n=\begin{bsmallmatrix}&&&-\varpi^{-n}\\&1\\&&1\\\varpi^n\end{bsmallmatrix}.
\end{equation}
Hence
\begin{equation}\label{paramodularKlingendecompositioneq2}
  \K{5}=\bigsqcup\limits_{y\in\OF/\p^5}
  \Kl{5}\begin{bsmallmatrix}1&&&y\varpi^{-5}\\&1\\&&1\\&&&1\end{bsmallmatrix}\sqcup\bigsqcup\limits_{z\in\OF/\p^4}
  \Kl{5}\begin{bsmallmatrix}1&&&z\varpi^{-4}\\&1\\&&1\\&&&1\end{bsmallmatrix}\begin{bsmallmatrix}&&&\varpi^{-5}\\&1\\&&1\\-\varpi^5\end{bsmallmatrix}.
\end{equation}
Using Lemma~\ref{Hcosetlemma2}, we get the required decomposition, which can be checked to be disjoint.
\end{proof}

\begin{proposition}\label{newminprop}
 \begin{align}\label{newminpropeq1}
  &f^\pm_{\rm new}=\sum_{u,v\in(\OF/\p)^\times}\sum_{\substack{x\in\OF/\p^2\\y\in\OF/\p^2}}\sigma_\chi^\pm(\begin{bsmallmatrix}1\\&&-1\\&1\\&&&1\end{bsmallmatrix}\begin{bsmallmatrix}1&&&y\varpi^{-5}\\&1\\&&1\\&&&1\end{bsmallmatrix}\begin{bsmallmatrix}1\\&1\\&x&1\\&&&1\end{bsmallmatrix}\begin{bsmallmatrix}uv\\&u\\&&v\\&&&1\end{bsmallmatrix}\begin{bsmallmatrix}1\\&\varpi\\&&\varpi^2\\&&&\varpi^3\end{bsmallmatrix})f^\pm_{\rm min}\nonumber\\
   &+\sum_{u,v\in(\OF/\p)^\times}\sum_{\substack{x\in\OF/\p\\y\in\OF/\p^2}}\sigma_\chi^\pm(\begin{bsmallmatrix}1&&&y\varpi^{-5}\\&1\\&&1\\&&&1\end{bsmallmatrix}\begin{bsmallmatrix}1\\&1\\&x\varpi&1\\&&&1\end{bsmallmatrix}\begin{bsmallmatrix}uv\\&u\\&&v\\&&&1\end{bsmallmatrix}\begin{bsmallmatrix}1\\&\varpi\\&&\varpi^2\\&&&\varpi^3\end{bsmallmatrix})f^\pm_{\rm min}\nonumber\\
  &+\sum_{u,v\in(\OF/\p)^\times}\sum_{\substack{x\in\OF/\p^2\\z\in\OF/\p}}\sigma_\chi^\pm(\begin{bsmallmatrix}&&&-\varpi^{-5}\\&&-1\\&1\\\varpi^5\end{bsmallmatrix}\begin{bsmallmatrix}1&&&z\varpi^{-4}\\&1\\&&1\\&&&1\end{bsmallmatrix}\begin{bsmallmatrix}1\\&1\\&x&1\\&&&1\end{bsmallmatrix}\begin{bsmallmatrix}uv\\&u\\&&v\\&&&1\end{bsmallmatrix}\begin{bsmallmatrix}1\\&\varpi\\&&\varpi^2\\&&&\varpi^3\end{bsmallmatrix})f^\pm_{\rm min}\nonumber\\
   &+\sum_{u,v\in(\OF/\p)^\times}\sum_{\substack{x\in\OF/\p\\z\in\OF/\p}}\sigma_\chi^\pm(\begin{bsmallmatrix}&&&-\varpi^{-5}\\&1\\&&1\\\varpi^5\end{bsmallmatrix}\begin{bsmallmatrix}1&&&z\varpi^{-4}\\&1\\&&1\\&&&1\end{bsmallmatrix}\begin{bsmallmatrix}1\\&1\\&x\varpi&1\\&&&1\end{bsmallmatrix}\begin{bsmallmatrix}uv\\&u\\&&v\\&&&1\end{bsmallmatrix}\begin{bsmallmatrix}1\\&\varpi\\&&\varpi^2\\&&&\varpi^3\end{bsmallmatrix})f^\pm_{\rm min}.
 \end{align}
\end{proposition}
\begin{proof}
Let $f_0$ be as in \eqref{f0defeq}. If a function $f$ in the standard space of $\pi_\chi$ is supported on $Hr\K{n}$, is right $\K{n}$-invariant, and satisfies $f(r)=c$, and if $Hr\K{n}=\bigsqcup_i Hrk_i$ for some representatives $k_i\in\K{n}$, then $f(g)=c\sum_if_0(gk_i^{-1}r^{-1})$ for all $g\in G$. Hence the result follows from Lemma~\ref{Hcosetlemma3}.
\end{proof}

\section{Matrix coefficients and local integrals}\label{s:localprelims}
In this section, we study matrix coefficients and various local integrals associated to test vectors in the  simple supercuspidal representations $\sigma_{\chi}^\pm$ defined in \eqref{chipmindeq}. Our test vectors will be (translates of) either the minimal vector $f^\pm_{\rm min}$  or the  paramodular newvector $f^\pm_{\rm new}$. The matrix coefficient for the minimal vector is computed in Sect. \ref{s:matrixc}, which leads to a formula for the formal degree of $\sigma_{\chi}^\pm$ in Proposition~\ref{p:formaldeg}. In Sect.~\ref{s:localwhittakerintegralminimal}, we compute the local Whittaker integral for the minimal vector, and in Sect.~\ref{s:localnov}, we  compute the local Novodvorsky integral for a certain diagonal translate of the minimal vector. We write down a formula for the matrix coefficient of the paramodular newvector  evaluated on the unipotent radical in Sect.~\ref{MC-unipotent} and we use this to compute the local Whittaker integral for the newvector in Sect.~\ref{s:localwhitnew}. In Sect.~\ref{locbessec} we compute the local Bessel integral of Gan--Gross--Prasad type for translates of the minimal vector.
\subsection{Inner product and matrix coefficients}\label{s:matrixc}
We define an inner product on $A$ (the standard model of $\pi_\chi=\cInd_H^G(\chi)$) by
\begin{equation}\label{innerproduct}
 \langle f_1,f_2\rangle=\vol(Z \bs H')^{-1}\int\limits_{Z\backslash G}f_1(x)\overline{f_2(x)}\,dx.
\end{equation}
This is well-defined, since the support of any $f\in A$ is contained in a subset of $G$ of the form $HC=ZK'C$ where $C$ is compact, and independent of the normalization of the Haar measure $dx$ on $Z \bs G$. This gives us a $G$-invariant  Hermitian pairing $(v_1, v_2) \mapsto \langle v_1,v_2\rangle$  on each of the representations $\sigma_\chi^+$, $\sigma_\chi^-$.

By definition, the matrix coefficient attached to the pair $(f_1,f_2)$ is the function
\begin{equation}\label{matrixcoeff}
 \Phi_{f_1,f_2}(g)=\langle\pi_\chi(g)f_1,f_2\rangle.
\end{equation}
Let $\Phi_{\rm min}^\pm$ be the matrix coefficient corresponding to the pair $(f^\pm_{\rm min},f^\pm_{\rm min})$, where $f^\pm_{\rm min}$ is the minimal vector defined in \eqref{fmindefeq}. Then
\begin{align}\label{Phiminfmineq}
 \Phi^\pm_{\rm min}(g)&=\langle\sigma_\chi^\pm(g) f^\pm_{\rm min},f^\pm_{\rm min}\rangle\nonumber\\
 &=\vol(Z \bs H')^{-1}\int\limits_{Z\backslash G}f^\pm_{\rm min}(xg)\overline{f^\pm_{\rm min}(x)}\,dx\nonumber\\
 &=\vol(Z \bs H')^{-1}\int\limits_{Z\backslash H'}f^\pm_{\rm min}(xg)\overline{\chi^\pm(x)}\,dx\nonumber\\
 &=\vol(Z \bs H')^{-1} \int\limits_{Z\backslash H'}f^\pm_{\rm min}(g)\,dx\nonumber\\
 &=f^\pm_{\rm min}(g).
\end{align}
Hence the minimal vector is its own matrix coefficient: $\Phi_{\rm min}^\pm=f_{\rm min}^\pm$. It is clear from the above and $G$-invariance that
\[
 \langle hf_{\rm min}^\pm, hf_{\rm min}^\pm\rangle=1 \qquad \text{for all } h \in G.
\]
For other vectors it is in general much more difficult to obtain explicit formulas for the matrix coefficient.

\begin{proposition}\label{p:formaldeg}Let the Haar measure $dx$ on $Z \bs G$ be normalized so that the volume of $Z \bs ZK$ equals 1. Then the formal degree of each of the representations $\sigma_\chi^+$ and $\sigma_\chi^{-}$ with respect to $dx$ equals $\frac{(q^4-1)(q^2-1)}{2}$. In other words, for any vector $v$ in the space of $\sigma_\chi^\pm$, we have $$\frac{(q^4-1)(q^2-1)}{2}\int_{Z \bs G}|\langle v, \sigma_\chi^\pm(x) v\rangle|^2 dx = \langle v, v \rangle^2.$$
\end{proposition}
\begin{proof}Since the quantity $R = \frac{\int_{Z \bs G}|\langle v, \sigma_\chi^\pm(x) v\rangle|^2 dx}{|\langle v, v \rangle|^2}$ does not depend on the vector $v$, we evaluate it for the vector $v=f_{\rm min}^\pm$. The calculation \eqref{Phiminfmineq} shows that $R = \vol(Z \bs H')$. It is clear from the definition that $\vol(Z \bs H') = 2 \ \vol (Z \bs H) = 2 \ [K: H \cap K]^{-1}$. By reducing modulo \(\p\) we see that
$[K:H\cap K]
  =
  [\GSp_4(\mathbf{k}):Z_{\GSp_4}(\mathbf{k})U(\mathbf{k})].
$
Here \(Z_{\GSp_4}(\mathbf{k})\) is the scalar center. Now
\[
  |U(\mathbf{k})|=q^4,\qquad |Z_{\GSp_4}(\mathbf{k})|=q-1,\qquad
  Z_{\GSp_4}(\mathbf{k})\cap U(\mathbf{k})=\{1\},
\]
and by a standard formula (see, for example, Chapter 8 of \cite{Tay92})
\[
  |\GSp_4(\mathbf{k})|=(q-1)|\Sp_4(\mathbf{k})|
  =(q-1)q^4(q^2-1)(q^4-1).
\]
Therefore
$[K:H\cap K]
  =
  (q^4-1)(q^2-1).
$
\end{proof}

\subsection{The Whittaker model and associated local integral}
Let $$B= \left\{\begin{bsmallmatrix}t_1&*&*&*\\&t_2&*&*\\&&\mu t_2^{-1}&*\\&&&\mu t_1^{-1}\end{bsmallmatrix}: t_1, t_2, \mu \in F^\times\right\}$$ be the  Borel subgroup of $G$ and $U$ be the unipotent radical of $B$. Thus we have
\[
 U = \left\{\begin{bsmallmatrix}1&a\\&1\\&&1&-a\\&&&1\end{bsmallmatrix}\begin{bsmallmatrix}1&&b&c\\&1&e&b\\&&1\\&&&1\end{bsmallmatrix}: a,b,c,e \in F \right\}.
\]
Our choice of Haar measure $dx$ on $F$ gives a Haar measure $du$ on $U$. Fix $c_1, c_2 \in \OF^\times$ and consider the character $\psi_{c_1, c_2}$ of $U$ defined by
\[
 \psi_{c_1, c_2} \left(\begin{bsmallmatrix}1&a&*&*\\&1&e&*\\&&1&-a\\&&&1\end{bsmallmatrix} \right) = \psi(c_1a + c_2e).
\]
An irreducible, admissible representation $\pi$ of $G$ is called generic if $\Hom_{U}(\pi, \psi_{c_1, c_2}) \ne 0$. This definition does not depend on the choice of $c_1$ or $c_2$. Define
\[
 \mathcal{W}(\psi_{c_1, c_2}) = \left\{W: G \rightarrow \C, \quad  W(ug) = \psi_{c_1, c_2}(u)W(g)\text{ for all } g\in G, u\in U  \right\}.
\]If $\pi$ is generic, then there exists a (unique) subspace $\mathcal{W}(\pi, \psi_{c_1, c_2}) \subset \mathcal{W}(\psi_{c_1, c_2})$, known as the Whittaker model for $\pi$, such that $\mathcal{W}(\pi, \psi_{c_1, c_2})$ gives a model for $\pi$ under the action of $G$ given by right translations.

For an irreducible, admissible, unitary representation $\pi$ of $G$, we define the quantity $J_0(v_1, v_2)$ for any two vectors $v_1$ and $v_2$ in $\pi$ by
\begin{equation}\label{J0-denf}
J_0(v_1, v_2) := \int^{\st}\limits_U \Phi_{v_1, v_2}(u) \psi_{c_1, c_2}^{-1}(u) du,
\end{equation}
where the symbol $\int^{\st}\limits_U$ denotes the stable integral (a form of regularization of a potentially non-convergent integral) in the sense of Lapid--Mao \cite[(2.1)]{LM15}. Note here that if $\pi$ is \emph{square-integrable} (e.g., a supercuspidal representation with unitary central character), then the integral \eqref{J0-denf} converges absolutely, and so we can replace the stable integral by the usual integral. We also define the normalized quantity
\begin{equation}\label{e:normalizedJ0}
 J_0(v) := \frac{J_0(v, v)}{\langle v, v \rangle},
\end{equation}
whose definition does not depend on the choice of $\langle\,, \rangle$.

The pairing $(v_1, v_2) \mapsto J_0(v_1, v_2)$ gives a Hermitian form on the space of $\pi$ that is $(U, \psi_{c_1, c_2})$ equivariant in $v_1$ and $(U, \psi_{c_1, c_2}^{-1})$  equivariant in $v_2$. It follows that if the pairing $(v_1, v_2) \mapsto J_0(v_1, v_2)$ is not identically 0, then $\pi$ must be generic. Conversely, if $\pi$ is generic, then using \cite[Prop 2.3]{LM15} we see that the pairing $(v_1, v_2) \mapsto J_0(v_1, v_2)$  descends to a non-degenerate pairing on a one-dimensional quotient of $\pi$. Therefore, $\pi$ is generic if and only if there is a non-zero vector $v$ in the space of $\pi$ such that $J_0(v) \neq 0$.

Furthermore, if $\pi$ is generic and $W: V_\pi \rightarrow \C$ is a non-zero $(U, \psi_{c_1, c_2})$ equivariant functional (such a functional is unique up to multiples), then there is a non-zero constant $c$ (which depends on the choice of $\langle \ , \ \rangle$ and the choice of the functional) such that $J_0(v_1, v_2) = c W(v_1) \overline{W(v_2)}$ for any two vectors $v_1$, $v_2$ in the space of $\pi$. We say that a vector $v$ in the space of $\pi$ is a \emph{test vector for the Whittaker functional} if $J_0(v) \neq 0$. We will refer to $J_0(v)$ as the \emph{local Whittaker integral} for $v$.

Let $\pi$ be an irreducible, admissible, generic representation of $G$ of trivial central character. Let $\chi$ be a character of $F^\times$, and let \(a(\chi)\) denote its conductor exponent;
equivalently, \(a(\chi)\) is the least non-negative integer \(n\) such that
\(\chi\) is trivial on \(1+\p^n\), with the convention \(1+\p^0=\OF^\times\). For any $W$ in $\mathcal{W}(\pi, \psi_{c_1, c_2})$ define the Novodvorsky zeta integral by
\begin{equation}\label{zeta-int-defn}
 Z(s, W, \chi) := \int\limits_{F^\times} \int\limits_F W(\begin{bsmallmatrix}\gamma\\&\gamma\\&x&1\\&&&1\end{bsmallmatrix})|\gamma|^{s-\frac 32} \chi(\gamma)\,dx\,d^\times \gamma.
\end{equation}
Recall the paramodular group $\K{n}$ defined in \eqref{paradefeq}.
Let the conductor of $\pi$ be $q^{a(\pi)}$. It was shown in \cite{NF} that $\pi$ has a vector fixed by $\K{n}$ if and only if $n \ge a(\pi)$; moreover, the space of $\K{a(\pi)}$-fixed vectors in $V_\pi$ is one-dimensional. A non-zero vector in the space of $\pi$ is said to be a (paramodular) \emph{newvector} if it is fixed by $\K{a(\pi)}$. If $W_{\rm new} \in \mathcal{W}(\pi, \psi_{c_1, c_2})$ is a newvector (in the Whittaker model of $\pi$) normalized by $W_{\rm new}(1)=1$, and $\chi$ is unramified, then by \cite[Thm.~7.5.4]{NF} we know that
\begin{equation}\label{e:localspinunram}
 Z(s, W_{\rm new}, \chi) = L(s, \pi \times \chi),
\end{equation}
where $L(s,\pi \times \chi)$ denotes the spinor (degree 4) $L$-factor of $\pi \times \chi$.

\subsection{The local Whittaker integral for the minimal vector} \label{s:localwhittakerintegralminimal}

Recall the definition of the local integral $J_0$ from \eqref{J0-denf}. We will show that $J_0$ is non-vanishing on certain translates of the minimal vector by diagonal matrices. \begin{proposition}\label{J0-non-vanishing}
We have $J_0( d_{\alpha,\beta} f^\pm_{\rm min},  d_{\gamma,\delta} f^\pm_{\rm min}) = 0$ unless all of the following conditions are satisfied,
$$
 \frac{\alpha}{\gamma}, \frac{\beta}{\delta} \in 1+\p,\quad \gamma \in \varpi^{-1}c_1^{-1}(1+\p) \text{ and } \delta \in \varpi^{-1}c_2^{-1}(1+\p).
$$
In case all the conditions above are satisfied, then $J_0( d_{\alpha,\beta} f^\pm_{\rm min},  d_{\gamma,\delta} f^\pm_{\rm min}) = q^{7}$. In particular \[J_0( d_{\varpi^{-1}c_1^{-1},\varpi^{-1}c_2^{-1}} f^\pm_{\rm min}) = q^7.\]
\end{proposition}
\begin{proof}
By \eqref{Phiminfmineq},
\begin{align*}
J_0( d_{\alpha,\beta} f^\pm_{\rm min},  d_{\gamma,\delta} f^\pm_{\rm min}) &= \int\limits_U \Phi_{ d_{\alpha,\beta} f^\pm_{\rm min},  d_{\gamma,\delta} f^\pm_{\rm min}}(u) \psi_{c_1, c_2}^{-1}(u)\,du \\
&= \int\limits_U \Phi^{\pm}_{{\rm min}}( d_{\gamma,\delta}^{-1} u d_{\alpha,\beta}) \psi_{c_1, c_2}^{-1}(u)\,du \\
&= \int\limits_U f^{\pm}_{{\rm min}}( d_{\gamma,\delta}^{-1} u  d_{\alpha,\beta}) \psi_{c_1, c_2}^{-1}(u)\,du.
\end{align*}
For $a,b,c,e \in F$, set
$$u(a,b,c,e) = \begin{bsmallmatrix}1&a\\&1\\&&1&-a\\&&&1\end{bsmallmatrix}\begin{bsmallmatrix}1&&b&c\\&1&e&b\\&&1\\&&&1\end{bsmallmatrix}.$$
Then
$$ d_{\gamma,\delta}^{-1} u(a,b,c,e)  d_{\alpha,\beta} = \begin{bsmallmatrix}\frac{\alpha^2\beta}{\gamma^2\delta} & \frac{\alpha \beta a}{\gamma^2\delta} & \frac{\alpha (b+ae)}{\gamma^2\delta} & \frac{(c+ab)}{\gamma^2\delta} \\ & \frac{\alpha \beta}{\gamma \delta} & \frac{\alpha e}{\gamma \delta} & \frac{b}{\gamma \delta} \\ && \frac{\alpha}{\gamma} & \frac{-a}{\gamma} \\ &&&1\end{bsmallmatrix}.$$
We need to determine when the above matrix lies in the support $H' = ZK' \sqcup g_\chi ZK'$ of $f^{\pm}_{{\rm min}}$.
Note that the top left $2 \times 2$ block of $g_\chi^{-1} d_{\gamma,\delta}^{-1} u(a,b,c,e)  d_{\alpha,\beta}$ is zero, which implies that $ d_{\gamma,\delta}^{-1} u(a,b,c,e)  d_{\alpha,\beta}$ never belongs to $g_\chi ZK'$. Hence, we see that $ d_{\gamma,\delta}^{-1} u(a,b,c,e)  d_{\alpha,\beta}$ belongs to the support of $f^{\pm}_{{\rm min}}$ if and only if it belongs to $K'$, and this happens if and only if
$$
 \frac{\alpha}{\gamma}, \frac{\beta}{\delta} \in 1+\p\quad\text{and}\quad\frac a{\gamma}, \frac e{\delta}, \frac b{\gamma \delta}, \frac c{\gamma^2\delta} \in \OF.
$$
Assuming these conditions, and using \eqref{J0-denf} and the definition \eqref{fmindefeq} of $f^{\pm}_{{\rm min}}$, we get
\begin{align*}
 &J_0( d_{\alpha,\beta} f^\pm_{\rm min},  d_{\gamma,\delta} f^\pm_{\rm min}) = \int\limits_{\substack{a \in \gamma \OF, \,b
\in \gamma \delta \OF \\ c \in \gamma^2 \delta \OF, \,e \in \delta \OF}} \psi\Big(\varpi^{-1}\Big(\frac{\alpha \beta a}{\gamma^2\delta}+\frac{\alpha e}{\gamma \delta}\Big)\Big) \psi(-(c_1a+c_2e))\,da\,de\,db\,dc \\
&\qquad= {\rm vol}(\gamma \delta \OF) {\rm vol}(\gamma^2 \delta \OF)\!\!\int\limits_{a \in \gamma \OF, \,e \in \delta \OF}\!\!\psi\Big(\frac a{\varpi}\Big(\frac{\alpha \beta}{\gamma^2\delta} - c_1\varpi\Big)\Big) \psi\Big(\frac e{\varpi}\Big(\frac{\alpha}{\gamma \delta} - c_2\varpi\Big)\Big)\,da\,de \\
&\qquad= |\gamma^4 \delta^3| \int\limits_{a \in  \OF, \,e \in  \OF} \psi(a(\varpi^{-1}- c_1\gamma)) \psi(e(\varpi^{-1}- c_2\delta))\,da\,de.
\end{align*}
The integral is non-zero if and only if
$$
 \gamma \in c_1^{-1}\varpi^{-1}(1+\p),\;\delta \in c_2^{-1}\varpi^{-1}(1+\p).
$$
If these conditions are satisfied, then
$$
 J_0( d_{\alpha,\beta} f^\pm_{\rm min},  d_{\gamma,\delta} f^\pm_{\rm min})   = |\varpi^{-7}| = q^{7},
$$
as asserted. For the final assertion, observe \eqref{J0-denf} and $\langle f_{\rm min}^\pm,f_{\rm min}^\pm\rangle=1$ by our choice of Haar measure on $Z\backslash G$.
\end{proof}

The above proposition allows us to obtain an explicit description of the Whittaker model of $\sigma^\pm_\chi$. For any $\phi \in \sigma^\pm_\chi$, define the function $W_\phi$ on $G$ by
\begin{align}\label{Wphidefeq}
W_\phi(g) &:= J_0(g \cdot \phi,  d_{\varpi^{-1}c_1^{-1}, \varpi^{-1}c_2^{-1}} f^\pm_{\rm min}) \nonumber\\
&= \int\limits_U \langle u g \phi,  d_{\varpi^{-1}c_1^{-1},\varpi^{-1} c_2^{-1}} f^\pm_{\rm min} \rangle \psi^{-1}_{c_1, c_2}(u)\,du \nonumber\\
&= \int\limits_U \langle  d_{\varpi c_1, \varpi c_2} u g \phi,  f^\pm_{\rm min} \rangle \psi^{-1}_{c_1, c_2}(u)\,du.
\end{align}
\begin{corollary}\label{Whittaker-model-defn-cor}
 The map $\phi \mapsto W_\phi$ is a non-zero intertwining map from $\sigma^\pm_\chi$ to $\mathcal{W}(\sigma^\pm_\chi, \psi_{c_1, c_2})$. Moreover,
 \[
  W_{d_{\alpha, \beta}f^\pm_{\rm min}}(1) = W_{f^\pm_{\rm min}}(d_{\alpha, \beta}) = \begin{cases} q^{7} & \text{ if } \alpha \in \varpi^{-1}c_1^{-1}(1+\p), \beta\in \varpi^{-1}c_2^{-1}(1+\p) \\ 0 &\text{ otherwise.} \end{cases}
 \]
\end{corollary}
\begin{proof}
A simple change of variables shows that
$$
 W_\phi(ug) = \psi_{c_1, c_2}(u) W_\phi(g).
$$
Hence $\phi \mapsto W_\phi$ is an intertwining map from $\sigma^\pm_\chi$ to $\mathcal{W}(\sigma^\pm_\chi, \psi_{c_1, c_2})$. To see that this map is non-zero, note from Proposition \ref{J0-non-vanishing} that
$$W_{f^\pm_{\rm min}}( d_{\varpi^{-1}c_1^{-1}, \varpi^{-1}c_2^{-1}}) \neq 0.$$ The formula for $W_{f^\pm_{\rm min}}(d_{\alpha, \beta})$ also follows from Proposition \ref{J0-non-vanishing}.
\end{proof}
\subsection{The Novodvorsky zeta integral for the minimal vector}\label{s:localnov}
For any $W$ in $\mathcal{W}(\sigma^\pm_\chi, \psi_{c_1, c_2})$, recall the defining formula \eqref{zeta-int-defn} for the Novodvorsky zeta integral.

\begin{proposition}\label{zeta-int-value-prop}
 Let $\alpha,\beta\in F^\times$. Let $\chi$ be a character of $F^\times$ with  $a(\chi)\in\{ 0, 1\}$. With the definition \eqref{Wphidefeq} of the Whittaker function, we have
 $$
  Z(s,  W_{d_{\alpha, \beta}f^\pm_{\rm min}}, \chi) = \begin{cases} (1-q^{-1})^{-1}q^{s+7/2} |\beta|^{1/2 -s} \chi(\beta\varpi c_2)^{-1}& \text{ if } \alpha \in \varpi^{-1}c_1^{-1}(1 + \p), \\ 0 & \text{ otherwise.}\end{cases}
 $$
\end{proposition}
\begin{proof}
Using \eqref{Wphidefeq}, we calculate that the integral $Z(s, W_{d_{\alpha, \beta}f^\pm_{\rm min}}, \chi)$ is equal to
\begin{align*}
& \int\limits_{F^\times} \int\limits_F W_{f^\pm_{\rm min}}(\begin{bsmallmatrix}\gamma\\&\gamma\\&x&1\\&&&1\end{bsmallmatrix}d_{\alpha, \beta})|\gamma|^{s-\frac 32} \chi(\gamma)\,dx\,d^\times \gamma \\
&= \int\limits_{F^\times} \int\limits_F \int\limits_U \langle  d_{\varpi c_1, \varpi  c_2} u \begin{bsmallmatrix}\alpha^2 \beta\gamma\\&\alpha\beta\gamma\\&\alpha\beta x&\alpha\\&&&1\end{bsmallmatrix} f^\pm_{\rm min},  f^\pm_{\rm min} \rangle \psi^{-1}_{c_1, c_2}(u)  |\gamma|^{s-\frac 32}\chi(\gamma)\,du\,dx\,d^\times \gamma \\
&= \int\limits_{F^\times} \int\limits_F \int\limits_U \Phi^{\pm}_{{\rm min}} \left( d_{\varpi c_1, \varpi  c_2} u \begin{bsmallmatrix}\alpha^2 \beta\gamma\\&\alpha\beta\gamma\\&\alpha\beta x&\alpha\\&&&1\end{bsmallmatrix}\right) \psi^{-1}_{c_1, c_2}(u)  |\gamma|^{s-\frac 32}\chi(\gamma)\,du\,dx\,d^\times \gamma \\
&= \int\limits_{F^\times} \int\limits_F \int\limits_U f^{\pm}_{{\rm min}} \left( d_{\varpi c_1, \varpi c_2} u \begin{bsmallmatrix}\alpha^2 \beta\gamma\\&\alpha\beta\gamma\\&\alpha\beta x&\alpha\\&&&1\end{bsmallmatrix}\right)\psi^{-1}_{c_1, c_2}(u)  |\gamma|^{s-\frac 32}\chi(\gamma)\,du\,dx\,d^\times \gamma \\
&= |\beta|^{1/2 - s}\chi(\beta)^{-1}\int\limits_{F^\times} \int\limits_F \int\limits_U f^{\pm}_{{\rm min}} \left( d_{\varpi c_1, \varpi c_2} u \begin{bsmallmatrix}\alpha^2 \gamma\\&\alpha \gamma\\&\alpha x&\alpha\\&&&1\end{bsmallmatrix}\right) \\
& \qquad \qquad \times \psi^{-1}_{c_1, c_2}(u)  |\gamma|^{s-\frac 32}\chi(\gamma)\,du\,dx\,d^\times \gamma,
\end{align*}
where in the last step we have made the substitutions $x \mapsto \beta^{-1}x$, $\gamma \mapsto \beta^{-1}\gamma$. We need to find when $ d_{\varpi c_1, \varpi c_2} u \begin{bsmallmatrix}\alpha^2 \gamma\\&\alpha\gamma\\&\alpha x&\alpha\\&&&1\end{bsmallmatrix} \in H' = ZK' \sqcup g_\chi ZK'$, the support of $f^{\pm}_{{\rm min}}$. Since the $(1,1)$ entry of $g_\chi^{-1} d_{\varpi c_1, \varpi c_2} u \begin{bsmallmatrix}\alpha^2 \gamma\\&\alpha\gamma\\&\alpha x&\alpha\\&&&1\end{bsmallmatrix}$ is equal to $0$, the matrices $d_{\varpi c_1, \varpi c_2} u \begin{bsmallmatrix}\alpha^2 \gamma\\&\alpha\gamma\\&\alpha x&\alpha\\&&&1\end{bsmallmatrix}$ are never in $g_\chi ZK'$. Writing out the matrix explicitly we see that the necessary and sufficient conditions for $ d_{\varpi c_1, \varpi c_2} u(a,b,c,e) \begin{bsmallmatrix}\alpha^2 \gamma\\&\alpha\gamma\\&\alpha x&\alpha\\&&&1\end{bsmallmatrix}$ to lie in  $ZK'$ are
$$
 \alpha \in \varpi^{-1}c_1^{-1}(1 + \p), \; x \in \p, \; \gamma \in \varpi^{-1}c_2^{-1}(1 + \p),\; a \in \p^{-1},\; b\in \p^{-2},\; c\in \p^{-3},\; e \in \p^{-1}.
$$
Hence, under the assumption that $\alpha \in \varpi^{-1}c_1^{-1}(1 + \p)$, we get that  $Z(s, W_{d_{\alpha, \beta}f^\pm_{\rm min}}, \chi)$ is equal to
\begin{align*}
 &|\beta|^{1/2 - s} \chi(\beta)^{-1}\int\limits_{\substack{\gamma \in \varpi^{-1}c_2^{-1}(1+\p) \\ x \in \p}} \:\int\limits_{\substack{a \in \p^{-1}, \ b\in \p^{-2} \\  c \in \p^{-3}, \ e \in \p^{-1} }} \psi(\varpi^2 c_1^2c_2 \alpha a \gamma + \varpi c_1c_2 \alpha e) \\
& \qquad \qquad  \psi(-(ac_1+ec_2)) |\gamma|^{s-\frac 32} \chi(\gamma)\, dx \,d^\times \,\gamma \,da \,db \,dc \,de \\
&= |\beta|^{1/2 - s} \chi(\beta)^{-1}q^{s-3/2} \chi(\varpi^{-1} c_2^{-1})\int\limits_{\substack{\gamma \in \varpi^{-1}c_2^{-1}(1+\p) \\ x \in \p}}\: \int\limits_{\substack{a \in \p^{-1}, \ b\in \p^{-2} \\  c \in \p^{-3}, \ e \in \p^{-1} }} \psi( c_1 a  +  c_2  e) \\
& \qquad \qquad  \psi(-(ac_1+ec_2))   dx \,d^\times \,\gamma \,da \,db \,dc \,de \\
&= |\beta|^{1/2 - s}\chi(\beta)^{-1} q^{s-3/2}\chi(\varpi^{-1} c_2^{-1})\int\limits_{\substack{\gamma \in \varpi^{-1}c_2^{-1}(1+\p) \\ x \in \p}} \:\int\limits_{\substack{a \in \p^{-1}, \ b\in \p^{-2} \\  c \in \p^{-3}, \ e \in \p^{-1} }}   \,dx \,d^\times \,\gamma \,da \,db \,dc \,de \\
&= (1-q^{-1})^{-1}q^{s+7/2} |\beta|^{1/2 -s}\chi(\beta\varpi c_2)^{-1}.
\end{align*}

This concludes the proof.
\end{proof}

\subsection{Matrix coefficient of the newvector evaluated on the unipotent radical}\label{MC-unipotent}
Recall from \eqref{fpmeq2} that the newvector $f^\pm_{\rm new}$ is supported on $H'd_{\varpi,\varpi}\K{5}$. We define the shifted newvector $\tilde f^\pm_{\rm new} := d_{\varpi, \varpi}f^\pm_{\rm new}$, i.e., for all $g\in G$ we have
\begin{equation}\label{shiftednewform}
 \tilde f^\pm_{\rm new}(g)=f^\pm_{\rm new}(gd_{\varpi, \varpi}).
\end{equation}
It is right-invariant under
\begin{equation}\label{tildeK5defeq}
 \dK{5}=d_{\varpi, \varpi}\,\K{5}d_{\varpi, \varpi}^{-1}=\{g\in G:\mu(g)=1\}\cap\begin{bsmallmatrix}\OF&\p&\p^2&\p^{-2}\\\p^4&\OF&\p&\p^2\\\p^3&\p^{-1}&\OF&\p\\\p^2&\p^3&\p^4&\OF\end{bsmallmatrix}
\end{equation}
and supported on $H'\dK{5}$. We will attempt to get some information about the matrix coefficient $\tilde\Phi_{\rm new}^\pm$ attached to the pair $(\tilde f^\pm_{\rm new},\tilde f^\pm_{\rm new})$. Let $S$ be the set of representatives from Lemma~\ref{Hcosetlemma3}, and let $\tilde S=Sd_{\varpi, \varpi}^{-1}$. The multipliers of the elements of $\tilde S$ are units.
We have $H'\dK{5}=\bigsqcup_{s\in\tilde S}H's$. Note that if $s,s'\in \tilde S$ are distinct, so that $H's$ and $H's'$ are disjoint, then the images of these sets in $Z\backslash G$ are also disjoint. Hence
\begin{align*}
 \tilde\Phi^\pm_{\rm new}(g)&=\vol(Z \bs H')^{-1}\int\limits_{Z\backslash G}\tilde f^\pm_{\rm new}(xg)\overline{\tilde f^\pm_{\rm new}(x)}\,dx\\
 &=\vol(Z \bs H')^{-1}\sum_{s\in \tilde S}\:\int\limits_{Z\backslash H's}\tilde f^\pm_{\rm new}(xg)\overline{\tilde f^\pm_{\rm new}(x)}\,dx\\
 &=\vol(Z \bs H')^{-1}\sum_{s\in \tilde S}\:\int\limits_{Z\backslash H'}\tilde f^\pm_{\rm new}(xsg)\overline{\tilde f^\pm_{\rm new}(xs)}\,dx\\
 &=\vol(Z \bs H')^{-1}\sum_{s\in \tilde S}\:\int\limits_{Z\backslash H'}\tilde f^\pm_{\rm new}(sg)\overline{\tilde f^\pm_{\rm new}(s)}\,dx\\
 &=\sum_{s\in \tilde S}\tilde f^\pm_{\rm new}(sg)\overline{\tilde f^\pm_{\rm new}(s)}\\
 &\stackrel{\eqref{fpmeq2}}{=}\sum_{s\in \tilde S}\tilde f^\pm_{\rm new}(sg)\\
 &=\sum_{s'\in\tilde S}\sum_{\substack{s\in \tilde S\\sg\in H's'}}\tilde f^\pm_{\rm new}(sg)\\
 &=\sum_{\substack{s,s'\in \tilde S\\sgs'{}^{-1}\in H'}}\chi^\pm(sgs'{}^{-1}).
\end{align*}
We see that, for $\tilde\Phi^\pm_{\rm new}(g)$ to be non-zero, there need to be $s,s'\in\tilde S$ such that $sgs'{}^{-1}\in H'$. For a given $s$, there can be at most one $s'$ satisfying this condition. Looking at multipliers, we see that
\begin{equation}\label{tildePhiformula2}
 \tilde\Phi^\pm_{\rm new}(g)=\begin{cases}
                              \displaystyle\sum_{\substack{s,s'\in \tilde S\\sgs'{}^{-1}\in H}}\chi(sgs'{}^{-1})&\text{if }v(\mu(g))\text{ is even},\\[6ex]
                              \displaystyle\pm\sum_{\substack{s,s'\in \tilde S\\sgs'{}^{-1}\in g_\chi H}}\chi(sgs'{}^{-1})&\text{if }v(\mu(g))\text{ is odd}.
                             \end{cases}
\end{equation}
These considerations show that
\begin{equation}\label{e:normnewvector}\langle \tilde f^\pm_{\rm new}, \tilde f^\pm_{\rm new} \rangle = \langle  f^\pm_{\rm new},  f^\pm_{\rm new} \rangle  = \tilde\Phi^\pm_{\rm new}(1) = |S| = (q-1)^2q^2(q+1)^2.
\end{equation}

In the following we use the notation $d = d_{\varpi, \varpi}$ for brevity. We require the values $\tilde\Phi^\pm_{\rm new}(g)$ for a unipotent matrix
\begin{equation}\label{tildePhicalceq1}
 g=\begin{bsmallmatrix}1&a\\&1\\&&1&-a\\&&&1\end{bsmallmatrix}\begin{bsmallmatrix}1&&b&c\\&1&e&b\\&&1\\&&&1\end{bsmallmatrix}
\end{equation}
with $a,b,c,e\in F$. Looking at Lemma~\ref{Hcosetlemma3}, we define
\begin{align}
 \label{tildePhicalceq4}s_1(u,v,x,y)&=d\begin{bsmallmatrix}uv\\&u\\&&v\\&&&1\end{bsmallmatrix}\begin{bsmallmatrix}1\\&1\\&x&1\\&&&1\end{bsmallmatrix}\begin{bsmallmatrix}1&&&y\varpi^{-5}\\&1\\&&1\\&&&1\end{bsmallmatrix}\begin{bsmallmatrix}1\\&&1\\&-1\\&&&1\end{bsmallmatrix}d^{-1}\\
 \label{tildePhicalceq5}s_2(u,v,x,y)&=d\begin{bsmallmatrix}uv\\&u\\&&v\\&&&1\end{bsmallmatrix}\begin{bsmallmatrix}1\\&1\\&x\varpi&1\\&&&1\end{bsmallmatrix}\begin{bsmallmatrix}1&&&y\varpi^{-5}\\&1\\&&1\\&&&1\end{bsmallmatrix}d^{-1}\\
 \label{tildePhicalceq6}s_3(u,v,x,y)&=d\begin{bsmallmatrix}uv\\&u\\&&v\\&&&1\end{bsmallmatrix}\begin{bsmallmatrix}1\\&1\\&x&1\\&&&1\end{bsmallmatrix}\begin{bsmallmatrix}1&&&y\varpi^{-4}\\&1\\&&1\\&&&1\end{bsmallmatrix}\begin{bsmallmatrix}&&&\varpi^{-5}\\&&1\\&-1\\-\varpi^5\end{bsmallmatrix}d^{-1}\\
 \label{tildePhicalceq7}s_4(u,v,x,y)&=d\begin{bsmallmatrix}uv\\&u\\&&v\\&&&1\end{bsmallmatrix}\begin{bsmallmatrix}1\\&1\\&x\varpi&1\\&&&1\end{bsmallmatrix}\begin{bsmallmatrix}1&&&y\varpi^{-4}\\&1\\&&1\\&&&1\end{bsmallmatrix}\begin{bsmallmatrix}&&&\varpi^{-5}\\&1\\&&1\\-\varpi^5\end{bsmallmatrix}d^{-1}
\end{align}
for $u,v\in\OF^\times$ and $x,y\in\OF$. If we set
\begin{align}
 \label{tildePhicalceq8}\tilde S_1&=\{s_1(u,v,x,y):u,v\in(\OF/\p)^\times,\:x,y\in\OF/\p^2\},\\
 \label{tildePhicalceq9}\tilde S_2&=\{s_2(u,v,x,y):u,v\in(\OF/\p)^\times,\:x\in\OF/\p,\:y\in\OF/\p^2\},\\
 \label{tildePhicalceq10}\tilde S_3&=\{s_3(u,v,x,y):u,v\in(\OF/\p)^\times,\:x\in\OF/\p^2,\:y\in\OF/\p\},\\
 \label{tildePhicalceq11}\tilde S_4&=\{s_4(u,v,x,y):u,v\in(\OF/\p)^\times,\:x,y\in\OF/\p\},
\end{align}
then $\tilde S=\bigsqcup_{i=1}^4\tilde S_i$. Since $g$ has the form (\ref{tildePhicalceq1}), one can verify by direct computation that
\begin{equation}\label{tildePhicalceq2}
 s\in\tilde S_i\text{ and }s'\in\tilde S_j\text{ with }i\neq j\;\Longrightarrow\;sgs'{}^{-1}\notin H.
\end{equation}
Hence, by \eqref{tildePhiformula2}, $\tilde\Phi^\pm_{\rm new}(g)=\sum_{i=1}^4\tilde\Phi^\pm_{{\rm new},i}(g)$ with
\begin{equation}\label{tildePhicalceq3}
 \tilde\Phi^\pm_{{\rm new},i}(g)=\sum_{\substack{s,s'\in \tilde S_i\\sgs'{}^{-1}\in H}}\chi(sgs'{}^{-1}).
\end{equation}
Consider the case $i=1$. The condition $s_1(u,v,x,y)gs_1(u',v',x',y')^{-1}\in H$ implies that
\begin{equation}\label{tildePhicalceq12}
 a\in\p^{-1},\quad b\in\OF,\quad c\in\p^{-2},\quad e\in\p,\quad a\varpi x+b\in\p.
\end{equation}
Conversely, if these conditions are satisfied, then we can set
\begin{equation}\label{tildePhicalceq13}
 u'=u,\quad v'=v,\quad x'=x-e\varpi^{-1},\quad y'=(ab+c)\varpi^{-2}+y,
\end{equation}
and find that
\begin{equation}\label{tildePhicalceq14}
 s_1(u,v,x,y)gs_1(u',v',x',y')^{-1}=\begin{bsmallmatrix}1&v(a\varpi x+b)\varpi^{-1}&-a\varpi u\\&1&&-a\varpi u\\&&1&-v(a\varpi x+b)\varpi^{-1}\\&&&1\end{bsmallmatrix}\in H.
\end{equation}
It follows that, assuming the first four conditions in \eqref{tildePhicalceq12} are satisfied,
\begin{align}\label{tildePhicalceq15}
 \tilde\Phi^\pm_{{\rm new},1}(g)&=\sum_{\substack{u,v\in(\OF/\p)^\times\\x,y\in\OF/\p^2\\a\varpi x+b\in\p}}\psi_0(v(a\varpi x+b)\varpi^{-1})\nonumber\\
 &=(q-1)q^2\sum_{\substack{v\in(\OF/\p)^\times\\x\in\OF/\p^2\\a\varpi x+b\in\p}}\psi(v(a\varpi x+b)\varpi^{-2}),
\end{align}
where we recall that $\psi_0(x)=\psi(x\varpi^{-1})$. While it is easy to calculate this further, we will refrain from doing so, because formula \eqref{tildePhicalceq15} is sufficient for the calculations in the following section.


Consider the case $i=2$. The condition $s_2(u,v,x,y)gs_2(u',v',x',y')^{-1}\in H$ implies that
\begin{equation}\label{tildePhicalceq17}
 a\in\OF,\quad b\in\OF,\quad c\in\p^{-2},\quad e\in\OF,\quad 1+ex\in\OF^\times.
\end{equation}
Conversely, if these conditions are satisfied, then we can set
\begin{equation}\label{tildePhicalceq18}
 u'=u(1+ex)^{-1},\quad v'=v(1+ex),\quad x'=x(1+ex)^{-1},\quad y'=(ab+c)\varpi^{2}+y,
\end{equation}
and find that
\begin{equation}\label{tildePhicalceq19}
 s_2(u,v,x,y)gs_2(u',v',x',y')^{-1}=\begin{bsmallmatrix}1&v(a-bx)&\frac{(b+ae)u}{1+ex}\\&1&\frac{eu}{v(1+ex)}&bu\\&&1&v(bx-a)\\&&&1\end{bsmallmatrix}\in H.
\end{equation}
It follows that, assuming the first four conditions in \eqref{tildePhicalceq17} are satisfied,
\begin{align}\label{tildePhicalceq20}
 \tilde\Phi^\pm_{{\rm new},2}(g)&=
 \sum_{\substack{u,v\in(\OF/\p)^\times\\x\in\OF/\p\\y\in\OF/\p^2\\1+ex\in\OF^\times}}\psi_0\Big(v(a-bx)+\frac{eu}{v(1+ex)}\Big)\nonumber\\
 &=q^2\sum_{\substack{u,v\in(\OF/\p)^\times\\x\in\OF/\p\\1+ex\in\OF^\times}}\psi\Big(v(a-bx)\varpi^{-1}+\frac{eu}{v(1+ex)}\varpi^{-1}\Big).
\end{align}
This can also be calculated further, but we will leave it at this stage and use it as input for the proof of Proposition \ref{J0-prop} in the following section.

Consider the case $i=3$. The condition $s_3(u,v,x,y)gs_3(u',v',x',y')^{-1}\in H$ implies that
\begin{equation}\label{tildePhicalceq21}
 a\in\p^{-2},\quad b\in\p^{-1},\quad c\in\p^{-3},\quad e\in\p,\quad b+a\varpi x\in\OF,\quad 1-(ab+c)\varpi^3y\in\OF^\times.
\end{equation}
Conversely, if these conditions are satisfied, then we can set
\begin{equation}\label{tildePhicalceq22}
 u'=u(1-(ab+c)\varpi^3y),\quad v'=v(1-(ab+c)\varpi^3y),\quad x'=x-e\varpi^{-1},\quad y'=\frac y{1-(ab+c)\varpi^3y},
\end{equation}
and find that
\begin{equation}\label{tildePhicalceq23}
 s_3(u,v,x,y)gs_3(u',v',x',y')^{-1}=\begin{bsmallmatrix} *&-\frac{v(b+a\varpi x)y}{1-(ab+c)\varpi^3y}&*&*\\ *&*&0&*\\ *&*&*&*\\-\frac{(ab+c)\varpi^4}{uv(1-(ab+c)\varpi^3y)^2}&*&*&*\end{bsmallmatrix}\in H.
\end{equation}
It follows that, assuming the first four conditions in \eqref{tildePhicalceq21} are satisfied,
\begin{align}\label{tildePhicalceq24}
 \tilde\Phi^\pm_{{\rm new},3}(g)&=\sum_{\substack{u,v\in(\OF/\p)^\times\\x\in\OF/\p^2\\y\in\OF/\p\\b+a\varpi x\in\OF\\1-(ab+c)\varpi^3y\in\OF^\times}}\psi_0\bigg(-\frac{v(b+a\varpi x)y}{1-(ab+c)\varpi^3y}-t\frac{(ab+c)\varpi^3}{uv(1-(ab+c)\varpi^3y)^2}\bigg)\nonumber\\
 &=\sum_{\substack{u,v\in(\OF/\p)^\times\\x\in\OF/\p^2\\y\in\OF/\p\\b+a\varpi x\in\OF\\1-(ab+c)\varpi^3y\in\OF^\times}}\psi\bigg(-\frac{v(b+a\varpi x)y\varpi^{-1}}{1-(ab+c)\varpi^3y}-u(ab+c)\varpi^2\bigg)\nonumber\\
 &=\sum_{\substack{u,v\in(\OF/\p)^\times\\x\in\OF/\p^2\\y\in\OF/\p\\b+a\varpi x\in\OF\\1-(ab+c)\varpi^3y\in\OF^\times}}\psi\Big(-v(b+a\varpi x)y\varpi^{-1}-u(ab+c)\varpi^2\Big).
\end{align}
This formula, which could be evaluated further, will serve as input in the proof of Proposition~\ref{J0-prop} below.

Consider the case $i=4$. The condition $s_4(u,v,x,y)gs_4(u',v',x',y')^{-1}\in H$ implies that
\begin{equation}\label{tildePhicalceq25}
 a\in\p^{-1},\quad b\in\p^{-1},\quad c\in\p^{-3},\quad e\in\OF,\quad1+ex\in\OF^\times,\quad 1-(ab+c)\varpi^3y\in\OF^\times.
\end{equation}
Conversely, if these conditions are satisfied, then we can set
\begin{alignat}{2}\label{tildePhicalceq26}
 &u'=u\frac{1-(ab+c)\varpi^3y}{1+ex},\qquad &&v'=v(1+ex)(1-(ab+c)\varpi^3y),\nonumber\\
 &x'=\frac x{1+ex},&&y'=\frac y{1-(ab+c)\varpi^3y},
\end{alignat}
and find that
\begin{equation}\label{tildePhicalceq27}
 s_4(u,v,x,y)gs_4(u',v',x',y')^{-1}=\begin{bsmallmatrix} *&-\frac{v(a-bx)y\varpi}{1-(ab+c)\varpi^3y}&*&*\\ *&*&\frac{eu}{v(1+ex)(1-(ab+c)\varpi^3y)}&*\\ *&*&*&*\\-\frac{(ab+c)\varpi^4}{uv(1-(ab+c)\varpi^3y)^2}&*&*&*\end{bsmallmatrix}\in H.
\end{equation}
It follows that, assuming the first four conditions in \eqref{tildePhicalceq25} are satisfied,
\begin{align}\label{tildePhicalceq28}
 \tilde\Phi^\pm_{{\rm new},4}(g)&=\sum_{\substack{u,v\in(\OF/\p)^\times\\x,y\in\OF/\p\\1+ex\in\OF^\times\\1-(ab+c)\varpi^3y\in\OF^\times}}
 \psi_0\bigg(-\frac{v(a-bx)y\varpi}{1-(ab+c)\varpi^3y}+\frac{eu}{v(1+ex)(1-(ab+c)\varpi^3y)}\nonumber\\
 &\hspace{45ex}-t\frac{(ab+c)\varpi^3}{uv(1-(ab+c)\varpi^3y)^2}\bigg)\nonumber\\
 &=\sum_{\substack{u,v\in(\OF/\p)^\times\\x,y\in\OF/\p\\1+ex\in\OF^\times\\1-(ab+c)\varpi^3y\in\OF^\times}}
 \psi\bigg(-\frac{v(a-bx)y}{1-(ab+c)\varpi^3y}+\frac{eu\varpi^{-1}}{v(1+ex)(1-(ab+c)\varpi^3y)}\nonumber\\
 &\hspace{45ex}-t\frac{(ab+c)\varpi^2}{uv(1-(ab+c)\varpi^3y)^2}\bigg).
\end{align}
This formula is difficult to evaluate explicitly, but the current form will serve as sufficient input in the proof of Proposition~\ref{J0-prop}.
\begin{remark}\label{rem:kloosterman-type}
Recall that \(\mathbf{k}=\OF/\p\), and let

  \[
  \psi_{\mathbf{k}}(\bar z)=\psi(\varpi^{-1}z), \qquad z\in \OF,
\]
be the induced non-trivial additive character of \(\mathbf{k}\). For \(A,B,C\in \mathbf{k}\), define  a two-dimensional hyper-Kloosterman sum
\[
  \mathcal K(A,B,C;\psi_\mathbf{k})
  =
  \sum_{u,v\in\mathbf{k}^\times}
  \psi_{\mathbf{k}}\left(Av+B\frac{u}{v}+C\frac{1}{uv}\right)= 
  \sum_{r,s\in \mathbf{k}^{\times}}
  \psi_{\mathbf{k}}\left(Ar+Bs+\frac{C}{r^2s}\right).
\]
Assume now that the first four conditions in \eqref{tildePhicalceq25} hold. For \(x,y\in \mathbf{k}\),
choose any lifts to~\(\OF\), and put
\[
  \Delta_y = 1-(ab+c)\varpi^3y, \qquad E_x = 1+ex.
\]
Whenever \(E_x,\Delta_y\in \OF^\times\), define
\[
  A_{x,y}
  =
  -\overline{\frac{\varpi(a-bx)y}{\Delta_y}},
  \qquad
  B_{x,y}
  =
  \overline{\frac{e}{E_x\Delta_y}},
  \qquad
  C_y
  =
  -\overline{\frac{t(ab+c)\varpi^3}{\Delta_y^2}},
\]
where the bar denotes reduction modulo \(\p\). These quantities are independent of
the chosen lifts of \(x\) and \(y\). With this notation, equation
\eqref{tildePhicalceq28} may be written in the cleaner form
\[
  \tilde\Phi^\pm_{{\rm new},4}(g)
  =
  \sum_{\substack{x,y\in \mathbf{k}\\ E_x\in\OF^\times,\ \Delta_y\in\OF^\times}}
  \mathcal K(A_{x,y},B_{x,y},C_y;\psi_{\mathbf{k}}).\]
\end{remark}

\subsection{The local Whittaker integral for the newvector}\label{s:localwhitnew}
Recall the definition \eqref{e:normalizedJ0} of the local Whittaker integral $J_0(v)$. Using the results of the previous section, we can evaluate $J_0( f^{\pm}_{{\rm new}})$ explicitly.

\begin{proposition}\label{J0-prop}
 We have
 $$
  J_0( f^{\pm}_{{\rm new}}) = q^5(1-q^{-2})^{-2}.
 $$
\end{proposition}
\begin{proof}
By change of variables and using \eqref{e:normnewvector}, we obtain 
\begin{align*}J_0(f^{\pm}_{{\rm new}}) &= \frac{1}{\langle f^{\pm}_{{\rm new}},  f^{\pm}_{{\rm new}}\rangle} \int^{\st}\limits_U \langle u f^{\pm}_{{\rm new}},  f^{\pm}_{{\rm new}}\rangle  \psi_{c_1, c_2}^{-1}(u) du,\\
&=\frac{q^7}{\langle f^{\pm}_{{\rm new}},  f^{\pm}_{{\rm new}}\rangle} \int^{\st}\limits_U \langle d_{\varpi, \varpi}^{-1} u d_{\varpi, \varpi} f^{\pm}_{{\rm new}},  f^{\pm}_{{\rm new}}\rangle  \psi_{c_1, c_2}^{-1}(d_{\varpi, \varpi}^{-1}ud_{\varpi, \varpi}) du, \\
&=\frac{q^7}{\langle f^{\pm}_{{\rm new}},  f^{\pm}_{{\rm new}}\rangle} \int^{\st}\limits_U \langle  u \tilde f^{\pm}_{{\rm new}},  \tilde f^{\pm}_{{\rm new}}\rangle\psi_{\frac{c_1}{\varpi}, \frac{c_2}{\varpi}}^{-1}(u) du,\\
&=\frac{q^5}{(q-1)^2(q+1)^2} \int^{\st}\limits_U \tilde \Phi^\pm_{{\rm new}}(u)  \psi_{\frac{c_1}{\varpi}, \frac{c_2}{\varpi}}^{-1}(u) du.
\end{align*}
From Sect.~\ref{MC-unipotent}, it is clear that $J_0(f^{\pm}_{{\rm new}}) = \frac{q^5}{(q-1)^2(q+1)^2} (J_{0, 1} + J_{0, 2} + J_{0, 3} + J_{0, 4})$, where, for $i\in\{1,2,3,4\}$,
$$
 J_{0,i} = \int\limits_U  \tilde\Phi^\pm_{{\rm new},i}(u) \psi_{\frac{c_1}{\varpi}, \frac{c_2}{\varpi}}^{-1}(u)\,du.
$$
Let us first compute $J_{0, 1}$. From (\ref{tildePhicalceq12}) and (\ref{tildePhicalceq15}), we have
\begin{align*}
 J_{0,1} &= \int\limits_{\substack{a \in \p^{-1}, b \in \OF \\ c \in \p^{-2}, e \in \p}} (q-1)q^2\sum_{\substack{v\in(\OF/\p)^\times\\x\in\OF/\p^2\\a\varpi x+b\in\p}}\psi(v(a\varpi x+b)\varpi^{-2}) \psi\Big(-\frac{c_1a+c_2e}{\varpi}\Big)\,da\,db\,dc\,de \\
 &= q^3(q-1)\int\limits_{a \in \p^{-1}, b \in \OF} \sum_{\substack{v\in(\OF/\p)^\times\\x\in\OF/\p^2\\a\varpi x+b\in\p}}\psi(v(a\varpi x+b)\varpi^{-2}) \psi\Big(-\frac{c_1a}{\varpi}\Big)\,da\,db\\
 &= q^3(q-1)\int\limits_{a \in \p^{-1}, b \in \p} \sum_{\substack{v\in(\OF/\p)^\times\\x\in\OF/\p^2}}\psi(vb\varpi^{-2}) \psi\Big(-\frac{c_1a}{\varpi}\Big)\,da\,db.
\end{align*}
Since $c_1 \in \OF^\times$, integration over the $a$ variable shows that $J_{0,1} = 0$.

Next we compute $J_{0,2}$.  From \eqref{tildePhicalceq17} and \eqref{tildePhicalceq20},
$$J_{0,2} = q^2 \int\limits_{\substack{a,b,e \in \OF \\ c \in \p^{-2}}} \sum_{\substack{u,v\in(\OF/\p)^\times\\x\in\OF/\p\\1+ex\in\OF^\times}} \psi(\varpi^{-1}va) \psi(-\varpi^{-1}vbx) \psi\Big(\frac{\varpi^{-1}eu}{v(1+ex)}\Big) \psi\Big(-\frac{c_1a+c_2e}{\varpi}\Big)\,da\,db\,dc\,de.
$$
Note that
\begin{align*}
\int\limits_{a \in \OF} \sum\limits_{v \in (\OF/\p)^\times} \psi(\varpi^{-1}a(v-c_1))\,da &= \int\limits_{\OF} \psi(\varpi^{-1}a(c_1-c_1)) da + \int\limits_{\OF} \sum\limits_{\substack{v \in (\OF/\p)^\times\\v \neq c_1}} \psi(\varpi^{-1}a(a-c_1))\,da \\
&= 1 + 0 = 1.
\end{align*}
Hence,
$$
 J_{0,2} = q^4 \int\limits_{b, e \in \OF} \sum_{\substack{u\in(\OF/\p)^\times\\x\in\OF/\p\\1+ex\in\OF^\times}} \psi(-\varpi^{-1}c_1bx) \psi\Big(\frac{\varpi^{-1}eu}{c_1(1+ex)}\Big) \psi\Big(-\frac{c_2e}{\varpi}\Big)\,db\,de.
$$
Similar to above,
\begin{align*}
\int\limits_{b \in \OF} \sum\limits_{\substack{x \in \OF/\p\\1+ex \in \OF^\times}} \psi(-\varpi^{-1}c_1bx)\,db &= \int\limits_{\OF} \psi(0)\,db + \int\limits_{b \in \OF} \sum\limits_{\substack{x \in (\OF/\p)^\times\\1+ex \in \OF^\times}} \psi(-\varpi^{-1}c_1bx) \,db\\
& = 1 + 0 = 1,
\end{align*}
so that
\begin{align*}
J_{0,2} &= q^4 \int\limits_{\OF} \sum\limits_{u \in (\OF/\p)^\times} \psi(\varpi^{-1}e(u/c_1-c_2))\,de \\
&= q^4 \Big(\int\limits_\OF  \psi(\varpi^{-1}e((c_1c_2)/c_1-c_2))\,de + \int\limits_\OF  \sum\limits_{\substack{u \in (\OF/\p)^\times\\u \neq c_1c_2}} \psi(\varpi^{-1}e(u/c_1-c_2))\,de\Big) \\
&= q^4(1+0) = q^4.
\end{align*}

Next we compute $J_{0,3}$. From \eqref{tildePhicalceq21} and \eqref{tildePhicalceq24},
$$
 J_{0,3} =\int\limits_{\substack{a \in \p^{-2}, b \in \p^{-1} \\ c \in \p^{-3}, e \in \p}} \sum_{\substack{u,v\in(\OF/\p)^\times\\x\in\OF/\p^2\\y\in\OF/\p\\b+a\varpi x\in\OF\\1-(ab+c)\varpi^3y\in\OF^\times}}\!\!\!\psi\Big(-\frac{v(b+a\varpi x)y}{\varpi}-u(ab+c)\varpi^2\Big) \psi\Big(-\frac{c_1a+c_2e}{\varpi}\Big) \,da\,db\,dc\,de.
$$
The integral over the variable $e$ produces a factor $q^{-1}$. With a change of variable $c \to c-ab$, followed by $b \to b-a\varpi x$ we get
$$
 J_{0,3} = q^{-1} \int\limits_{\substack{a \in \p^{-2}, b \in \OF \\ c \in \p^{-3}}} \sum_{\substack{u,v\in(\OF/\p)^\times\\x\in\OF/\p^2\\y\in\OF/\p\\1-c\varpi^3y\in\OF^\times}}\psi(-vby\varpi^{-1}-uc\varpi^2) \psi\Big(-\frac{c_1a}{\varpi}\Big)\,da\,db\,dc.
$$
Integration over the $a$ variable shows that $J_{0,3} = 0$.

Finally, we will compute $J_{0,4}$. From \eqref{tildePhicalceq25} and \eqref{tildePhicalceq28},
\begin{align*}
J_{0,4} &= \int\limits_{\substack{a, b \in \p^{-1}\\e \in \OF, c \in \p^{-3}}} \sum_{\substack{u,v\in(\OF/\p)^\times\\x,y\in\OF/\p\\1+ex\in\OF^\times\\1-(ab+c)\varpi^3y\in\OF^\times}}\psi\bigg(-\frac{v(a-bx)y}{1-(ab+c)\varpi^3y}+\frac{eu\varpi^{-1}}{v(1+ex)(1-(ab+c)\varpi^3y)}\nonumber\\
 &\hspace{33ex}-t\frac{(ab+c)\varpi^2}{uv(1-(ab+c)\varpi^3y)^2}\bigg)\psi\Big(-\frac{c_1a+c_2e}{\varpi}\Big)\,da\,db\,dc\,de.
\end{align*}
A change of variable $c \to c-ab$ leads to
\begin{align*}
 J_{0,4} &= \int\limits_{\substack{a, b \in \p^{-1}\\e \in \OF, c \in \p^{-3}}} \sum_{\substack{u,v\in(\OF/\p)^\times\\x,y\in\OF/\p\\1+ex\in\OF^\times\\1-c\varpi^3y\in\OF^\times}}\psi\bigg(-\frac{v(a-bx)y}{1-c\varpi^3y}+\frac{eu\varpi^{-1}}{v(1+ex)(1-c\varpi^3y)}\nonumber\\
 &\hspace{35ex}-t\frac{c\varpi^2}{uv(1-c\varpi^3y)^2}\bigg)\psi\Big(-\frac{c_1a+c_2e}{\varpi}\Big)\,da\,db\,dc\,de.
\end{align*}
The integral contains the factor
$$
 \int\limits_{a \in \p^{-1}} \psi\Big(-\frac a{\varpi}\Big(\frac{vy\varpi}{1-c\varpi^3y}+c_1\Big)\Big)\,da = \int\limits_{a \in \p^{-1}} \psi\Big(-\frac a{\varpi}\Big)\,da = 0,
$$
so that $J_{0,4} = 0$. Putting all this together, we see that 
$$J_0(f^{\pm}_{{\rm new}}) = \frac{q^5}{(q-1)^2(q+1)^2} (0 + q^4 + 0 + 0),$$
and hence we get the proposition.
\end{proof}
\subsection{The Bessel model and associated local integral}
Let $S = \mat{a}{b/2}{b/2}{c}$ with $a,b,c \in F$ and $d = b^2-4ac \neq 0$. Set
\begin{align*}
 T_S &:= \{g \in \GL_2(F):\, {}^tg S g = {\rm det}(g) S\}\\
 &\phantom{;}= \left\{\mat{x+by/2}{cy}{-ay}{x-by/2} :\, x, y \in F, x^2 - dy^2/4 \neq 0\right\}.
\end{align*}
Let $L = F(\sqrt{d})$ if $d \not\in (F^\times)^2$ and $L = F \oplus F$ otherwise. The map $\phi_S:T_S \to L^\times$, given by
\begin{equation}\label{e:TSiso}T_S \ni \mat{x+by/2}{cy}{-ay}{x-by/2} \longmapsto \begin{cases} x+y\sqrt{d}/2 & \text{ if } L \text{ is a field,} \\ (x+y\sqrt{d}/2, x-y\sqrt{d}/2) & \text{ if } L = F \oplus F,
\end{cases}
\end{equation}
is a group isomorphism. We embed $T_S$ in $G$ via
\[
 T_S \ni g \longmapsto \mat{g}{}{}{{\rm det}(g) g'} \text{ where } g' := \mat{}{1}{1}{} {}^tg^{-1} \mat{}{1}{1}{}.
\]
Note that if $\tilde S=\lambda \,^t\!ASA$ for some $\lambda \in F^\times$ and $A \in \GL_2(F)$, then $T_{\tilde S} = A^{-1} T_S A \simeq L^\times$.
Let $N$ be the unipotent radical of the Siegel parabolic subgroup given by
\[
 N = \{\begin{bsmallmatrix}1&&u&z\\&1&w&u\\&&1\\&&&1\end{bsmallmatrix} :\, u, w, z \in F\}.
\]
Let $\theta_S$ be the character of $N$ given by
\begin{equation}\label{thetaSdef}
 \theta_S(\begin{bsmallmatrix}1&&u&z\\&1&w&u\\&&1\\&&&1\end{bsmallmatrix}) := \psi({\rm tr}(S  \mat{u}{z}{w}{u}\mat{}{1}{1}{})) = \psi(az+bu+cw).
\end{equation}
Let $\Lambda$ be any character of $L^\times$ such that $\Lambda |_{F^\times} = 1$. We identify $\Lambda$ with a character of $T_S$ using the isomorphism \eqref{e:TSiso}; more precisely, we let $\Lambda_S$ be the character of $T_S$ given by
\begin{equation}
 \Lambda_S(t) := \Lambda(\phi_S(t)),\qquad t\in T_S.
\end{equation}
Note that if $\tilde S=\lambda \,^t\!ASA$ for some $\lambda \in F^\times$ and $A \in \GL_2(F)$, then $\Lambda_{\tilde S}(A^{-1} t A) = \Lambda_S(t)$.

For an irreducible, admissible, unitary, tempered representation $(\pi,V_\pi)$ of trivial central character and vectors $v_1, v_2, v \in V_\pi$, with $v \neq 0$, define
\begin{equation}\label{Besseldefn}
 B_{\Lambda, \theta_S}(v_1, v_2) := \int\limits_{F^\times \backslash T_S} \int\limits_N^{\st} \Phi_{v_1,v_2}(tn) \Lambda_S^{-1}(t) \theta_S^{-1}(n)\,dn\,dt,
\end{equation}
\begin{equation}\label{Bessel-fnal-defn}
 B_{\Lambda, \theta_S}(v) := \frac{B_{\Lambda, \theta_S}(v,v)}{\langle v, v\rangle}.
\end{equation}
The representation $\pi$ is said to have a $(S, \Lambda)$-Bessel model if $\Hom_{T_SN}(\pi,\C_{\Lambda \otimes \theta_S}) \neq 0$, in which case the space is known to be one-dimensional. It follows from \cite[Prop.~5.7]{walds12} that $\pi$ has  a $(S, \Lambda)$-Bessel model if and only if there exists $v_1, v_2$ such that $B_{\Lambda, \theta_S}(v_1, v_2) \neq 0$, in which case the pairing $(v_1, v_2) \mapsto B_{\Lambda, \theta_S}(v_1, v_2)$  descends to a non-degenerate pairing on a one-dimensional quotient of $\pi$. Therefore, $\pi$ has  a $(S, \Lambda)$-Bessel model if and only if there is a non-zero vector $v$ in the space of $\pi$ such that $B_{\Lambda, \theta_S}(v) \neq 0$, in which case $v$ is said to be a $(S, \Lambda)$-test vector for $\pi$. We will refer to $B_{\Lambda, \theta_S}(v)$ as the \emph{local Bessel integral} of type $(S, \Lambda)$ for $v$.

Suppose that $\tilde S=\lambda \,^t\!ASA$ for some $\lambda \in F^\times$ and $A \in \GL_2(F)$. A straightforward calculation verifies that
\begin{equation}\label{Btrafoeq}
 B_{\Lambda, \theta_S}(v')  = |\lambda\det(A)|^3 \ B_{\Lambda, \theta_{\tilde S}}(v), \quad \text{ where } v'=\pi(\mat{\lambda A}{}{}{A'})v.
\end{equation}
Therefore, in order to compute the local Bessel integral, we may replace $S$ by $\tilde S$ (for a suitable $\lambda$ and $A$) at the cost of changing the vector $v$ by a translate. Clearly, $\pi$ has  a $(S, \Lambda)$-Bessel model if and only if it has a $(\tilde S, \Lambda)$-Bessel model. In particular, the question of whether $\pi$ has a $(S, \Lambda)$-Bessel model  depends only on $L$ and $\Lambda$ and not on the particular choice of the matrix $S$ such that $T_S \simeq L^\times$.

\subsection{The local Bessel integral for the minimal vector}\label{locbessec}

In this subsection we specialize the general setup of the previous subsection to the case
\[
  S=S_a:=\mat{a}{}{}{1},
  \qquad -a\in\OF \text{ a non-square}.
\]
Thus \(L=F(\sqrt{-a})\) is a field extension, and throughout this subsection
\(T_S,\theta_S,\Lambda_S\) refer to this fixed matrix \(S=S_a\). For elements $\alpha, \beta \in F^\times$ define
$$
 f^{\alpha, \beta}_{\rm min} := d_{\alpha, \beta}^{-1} \cdot f^{\pm}_{\rm min}.
$$
Given a character $\Lambda$ of $L^\times$ such that $\Lambda|_{F^\times} = 1$, we want to compute $B_{\Lambda, \theta_S}(f^{\alpha, \beta}_{\rm min})$  for suitable values of $\alpha, \beta$ such that $B_{\Lambda, \theta_S}(f^{\alpha, \beta}_{\rm min}) \neq 0$.

Define the non-negative integer
\begin{equation}\label{m0def}
 m_0 = m_0(\Lambda,a) := {\rm min}\{m\ge 0 : \Lambda |_{1+\varpi^m \OF[\sqrt{-a}]}= 1\}.
\end{equation}

Suppose that $m_0 \ge 2$. Then, using the fact that  $\Lambda|_{\OF^\times} = 1$ (since $\Lambda|_{F^\times} = 1$), it is easy to show that the map $y \mapsto \Lambda(1 + \varpi^{m_0 -1}y \sqrt{-a})$ is a \emph{non-trivial} additive character on $\OF$ that is trivial on $\p$. Hence there exists a unit $u_0=u_0(\Lambda, \psi, a) \in \OF^\times$ such that
\begin{equation}\label{u0def}
 \Lambda(1 + \varpi^{m_0 -1}y \sqrt{-a})= \psi(\varpi^{-1}u_0y), \quad y \in \OF
\end{equation}
Before stating the result, we fix a normalization of the quotient measure on \(F^\times\bs T_S\)
for this particular choice of \(S\). Since \(S=\mat{a}{}{}{1}\), the isomorphism
\eqref{e:TSiso} implies that
\[
  \phi_S^{-1}(1+y\sqrt{-a})=\mat{1}{y}{-ay}{1}.
\]
The map
\[
  y\longmapsto F^\times\phi_S^{-1}(1+y\sqrt{-a})
\]
identifies \(F\) with the open subset of
\(F^\times\bs T_S\simeq F^\times\bs L^\times\) represented by elements with
non-zero \(F\)-part whose complement has measure zero. We normalize the
Haar measure \(dt\) on \(F^\times\bs T_S\) by requiring that, for every locally
constant function \(\mathcal F\) on \(F^\times\bs T_S\),
\[
  \int_{F^\times\bs T_S}\mathcal F(t)\,dt
  =
  \int_F
  \mathcal F\!\left(\phi_S^{-1}(1+y\sqrt{-a})\right)
  |1+ay^2|^{-1}\,dy,
\]
where \(dy\) is the additive Haar measure on \(F\) normalized by
\(\vol(\OF)=1\). With this normalization, the image of
\(1+\OF\sqrt{-a}\) in \(F^\times\bs L^\times\) has volume \(1\), since
\(|1+ay^2|=1\) for \(y\in\OF\).

\begin{remark} If \(L/F\)
is an inert quadratic field extension and the residual characteristic is odd, then \(-a\)
is a non-square unit and with the choice of measure above we have
\[
  \vol(\OF^\times\bs\OF_L^\times)
  =
  \int_F |1+ay^2|^{-1}\,dy
  =
  1+\sum_{r\geq 1}\vol(\varpi^{-r}\OF^\times)q^{-2r}
  =
  1+q^{-1}
  =
  \frac{\zeta_F(1)}{\zeta_F(2)}.
\]
\end{remark}

We now state our result.
\begin{proposition}\label{Bessel-model-prop}
Assume that the residual characteristic of \(F\) is odd. Let
\(S=\mat{a}{}{}{1}\), where \(-a\in\OF\) is a non-square, let
\(L=F(\sqrt{-a})\), and let \(\Lambda\) be a character of \(L^\times\) such that
\(\Lambda|_{F^\times}=1\). Let the integer \(m_0\) be defined as in
\eqref{m0def} and suppose that \(m_0\geq 2\) and \(2m_0-3\geq v(a)\).
Let \(u_0\in\OF^\times\) be as in \eqref{u0def}. Then, with the above
normalization of the Haar measure on \(F^\times\bs T_S\), for
\(\alpha\in\varpi^{1-m_0}u_0(1+\p)\) and \(\beta\in\varpi(1+\p)\), we have
\[
  B_{\Lambda,\theta_S}(d_{\alpha,\beta}^{-1}\cdot f^\pm_{\rm min})
  =
  q^{-4m_0+7}.
\]
\end{proposition}
\begin{proof}
By the definitions \eqref{Besseldefn} and \eqref{Bessel-fnal-defn},
\begin{align*}
B_{\Lambda, \theta_S}(f^{\alpha, \beta}_{\rm min}) &= \int\limits_{F^\times \backslash T_S} \int\limits_N^{\st} \Phi_{f^{\alpha, \beta}_{\rm min}, f^{\alpha, \beta}_{\rm min}}(nt) \Lambda_S^{-1}(t) \theta_S^{-1}(n)\,dn\,dt \\
&= \int\limits_{F^\times \backslash T_S} \int\limits_N^{\st} f^{\pm}_{\rm min}(d_{\alpha, \beta} n t d_{\alpha, \beta}^{-1}) \Lambda_S^{-1}(t) \theta_S^{-1}(n)\,dn\,dt \\
&= \int\limits_{y \in F} \int\limits_{u,w,z \in F}^{\st}  f^{\pm}_{\rm min}(d_{\alpha, \beta} \begin{bsmallmatrix}1&&u&z\\&1&w&u\\&&1\\&&&1\end{bsmallmatrix}  \begin{bsmallmatrix}1&y\\-ya&1\\&&1&-y\\&&ya&1\end{bsmallmatrix}  d_{\alpha, \beta}^{-1}) \\
& \qquad \qquad \times \Lambda^{-1}(1+y\sqrt{-a}) \psi^{-1}(az+w) |1 + ay^2|^{-1}\,dy\,du\,dw\,dz.
\end{align*}

We need to check when $d_{\alpha, \beta} n t d_{\alpha, \beta}^{-1} \in H' = ZK' \sqcup g_\chi ZK'$, the support of $f^{\pm}_{\rm min}$. Since the $(1,1)$ entry of $g_\chi^{-1} d_{\alpha, \beta} n t d_{\alpha, \beta}^{-1}$ is $0$, we see that $d_{\alpha, \beta} n t d_{\alpha, \beta}^{-1}$ is never in $g_\chi ZK'$. One can check that $d_{\alpha, \beta} n t d_{\alpha, \beta}^{-1} \in ZK'$ if and only if
$$
 y \in \p^{m_0-1},\; u \in \p^{m_0-2},\; z \in \p^{2m_0-3},\;  w \in \p^{-1}.
$$
Hence
\begin{align*}
&B_{\Lambda, \theta_S}(f^{\alpha, \beta}_{\rm min})\\
&= \int\limits_{y \in \p^{m_0-1}} \int\limits_{\substack{u \in  \p^{m_0-2} \\  z \in \p^{2m_0-3} \\  w \in \p^{-1}}} \psi(\varpi^{-1}(\alpha y + \beta w + \beta a u y)) \psi^{-1}(az+w) \Lambda^{-1}(1+y\sqrt{-a})\,dy\,du\,dw\,dz \\
&= q^{3-2m_0}\int\limits_{y \in \p^{m_0-1}} \int\limits_{\substack{u \in  \p^{m_0-2} \\   w \in \p^{-1}}} \psi(\varpi^{-1}(\alpha y + \beta w)) \psi^{-1}(w) \Lambda^{-1}(1+y\sqrt{-a})\,dy\,du\,dw\\&= q^{5-3m_0}\int\limits_{y \in \p^{m_0-1}} \int\limits_{   w \in \p^{-1}} \psi(\varpi^{-1}(\alpha y + \beta w)) \psi^{-1}(w) \Lambda^{-1}(1+y\sqrt{-a})\,dy\,dw\,
\end{align*}
where we used $\beta a u y \in  \p$ and $\int\limits_{z \in \p^{2m_0-3}} \psi^{-1}(  z) dz = \mathrm{vol}(\p^{2m_0-3}) = q^{3-2m_0}.$ 

The integral over the $w$ variable gives $\int\limits_{ w \in \p^{-1}} \psi(w(\beta\varpi^{-1}- 1))\,dw =  q$, because $\beta\varpi^{-1}- 1 \in \p$ by hypothesis.
Finally, the integral in the $y$ variable is now equal to 
\begin{align*}
 &\int\limits_{ y \in \p^{m_0-1}} \psi(\varpi^{-1}\alpha y) \Lambda^{-1}(1+y\sqrt{-a})\,dy \\
 &\qquad= q^{1-m_0} \int\limits_{y \in \OF} \psi(\varpi^{-1}u_0y) \Lambda^{-1}(1+\varpi^{m_0-1}y\sqrt{-a})\,dy\\
 &\qquad=q^{1-m_0}.
\end{align*}
Putting it all together, we obtain $B_{\Lambda, \theta_S}(f^{\alpha, \beta}_{\rm min}) = q^{7-4m_0}$
 as required.
\end{proof}



\section{Explicit global period formulas and applications}\label{s:global}
In this section, we demonstrate how our local results enter into global applications.
\subsection{Basic global notations}\label{notations-sect}
For a commutative ring $R$, we let $G(R):=\GSp_4(R)$.  In this section, we will work in the setup of automorphic forms and representations over $G(\A)$, where $\A$ denotes the ring of adeles over~$\Q$.

Given an automorphic representation $\pi$ of $G(\A)$, the global $L$-functions denoted by $L(s, \pi)$ include the archimedean factors, so that  we have an Euler factor decomposition $L(s, \pi) = \prod_v L(s, \pi_v)$ with $v$ ranging over all the places of $\Q$. For finite set of places $S$ of $\Q$, we
use the notation $L^S(s, \pi):=\prod_{v \notin S}L(s, \pi_v)$ for
 the partial $L$-function obtained by omitting the factors corresponding to
the places in $S$.

We let $\psi$ denote the standard non-trivial additive character of $\Q \bs \A$ that is unramified at all finite places and equals  $e^{2 \pi i x}$ at $\R$. We define the character $\psi_U$ of $U(\Q)\bs U(\A)$ by
\[
 \psi_U \left(\begin{bsmallmatrix}1&a&*&*\\&1&e&*\\&&1&-a\\&&&1\end{bsmallmatrix} \right) = \psi(-a - e).
\]

We let $K_\infty$ be the maximal compact subgroup of $\Sp_4(\R)$ that fixes the point $iI_2$. For each finite prime $p$, put $K_p = G(\Z_p)$. We fix the measure on $\A^\times G(\Q)\bs G(\A)$ to be the Tamagawa measure (which gives it volume equal to 2). We take the usual Lebesgue measure on $\R$, which gives us a measure on $U(\R)$. We take the Haar measure on $U(\A)$ to be the product measure (recall that we fixed measures on $U(\Q_p)$ in Sect.~\ref{s:localprelims}). We obtain a Haar measure on $U(\Q) \bs U(\A)$ by giving $U(\Q)$ the counting measure; it can be checked that $U(\Q) \bs U(\A)$ has volume~$1$. We note that this choice of measures is compatible with the choices in \cite{chen-ichino} and \cite{LM15}.

Given measurable functions $\phi_i: \A^\times G(\Q)\bs G(\A) \rightarrow \C$ for $i=1,2$, we define the Petersson inner product
$$
 \langle \phi_1, \phi_2 \rangle = \int\limits_{\A^\times G(\Q)\bs G(\A)} \phi_1(g) \overline{\phi_2(g)} \,dg
$$
whenever this integral converges.

The notation
$A \ll_{x,y,z} B$ or $A = O_{x,y,z}(B)$ will mean there exists a positive
constant $C$ depending at most on $x, y, z$ such that $|A| \le C |B|$.  We use $A \asymp_{x,..y} B$ to mean that $A \ll_{x,..y} B$ and $B \ll_{x,..y} A$.
 The symbol $\epsilon$ will denote a small positive quantity whose value may change from line to line.

\subsection{An explicit relation between the Whittaker and \texorpdfstring{$L^2$}{}-normalizations}
Let $\pi = \otimes_v \pi_v$ be an irreducible, unitary, cuspidal automorphic representation of $G(\A)$ with trivial central character. We further assume that $\pi$ is globally generic, i.e.,  for each non-zero $\phi$ in the space of $\pi$ the function
\begin{equation}\label{global-whittaker-defn}
 W_\phi(g):= \int\limits_{U(\Q)\bs U(\A)}\phi(ug)\psi_U^{-1}(u)\,du
\end{equation}
is non-zero. It is then known (see Section 1.1 of \cite{chen-ichino} or Section 1.1 of \cite{Schmidt2016}) that $\pi$ is not CAP and that $\pi$ has a global functorial transfer to an automorphic representation $\Pi$ of $\GL_4(\A)$; we say that $\pi$ is of general type if $\Pi$ is cuspidal and we say that $\pi$ is endoscopic otherwise.

Given $\phi$ in the space of $\pi$ such that $W_\phi(1)$ is non-zero,  it is of considerable interest to understand the quantity $\frac{|W_\phi(1)|^2}{\langle \phi, \phi \rangle}$. This ratio quantifies  the difference between the arithmetic/Whittaker normalization of $\phi$ (the first Whittaker coefficient being made equal to 1) and the $L^2$-normalization of $\phi$ (the Petersson norm of $\phi$ being made equal to 1) which is crucial for various analytic and arithmetic applications. More generally, if $W_\phi(1) = 0$, one can choose some $g_0 \in G(\A)$ such that  $W_\phi(g_0)\neq 0$ and try to understand $\frac{|W_\phi(g_0)|^2}{\langle \phi, \phi \rangle}$. Lapid and Mao made the following remarkable conjecture in \cite{LM15}.

\begin{conjecture}\label{lapid-mao}
 Let $\pi\simeq\otimes_v \pi_v$ be an irreducible, unitary, cuspidal, generic automorphic representation of $G(\A)$ with trivial central character. Let $\phi$ be a vector in the space of $\pi$ corresponding to $\otimes_v \phi_v$. Let $g_0 = (g_{0,v})_v \in G(\A)$. Let $S$ be a set of places including the place at infinity such that for all $p \notin S$, $\pi_p$ and $\phi_p$ are unramified and $g_{0,p} \in K_p$. Then we have
 $$
   \frac{|W_\phi(g_0) |^2   }{\langle \phi, \phi\rangle} =  2^{-c}\frac{\zeta^{S} (2) \zeta^{(S)}(4) }{L^{S}(1,  \pi, \Ad)}  \prod_{v \in S}J_0(g_{0,v} \cdot \phi_v),
 $$
 where $J_0(g_{0,v} \cdot \phi_v)$ is defined in \eqref{e:normalizedJ0} for $v$ non-archimedean and as in \cite[Sect.~2.5]{LM15} for $v=\infty$, $L^{S}(1,  \pi, \Ad)$ denotes the adjoint (degree 10) $L$-function of $\pi$ with the factors in $S$ omitted and $$c = \begin{cases} 1 & \text{ if } \pi \text{ is of general type,} \\ 2 & \text{ if } \pi \text{ is endoscopic.}\end{cases}$$
 \end{conjecture}

Recently, Furusawa and Morimoto \cite[Theorem 6.3]{FM22} have proved the above conjecture assuming that $\pi$ is tempered. They have also proved \cite[Corollary 8.1]{FM22} that $\pi$ is tempered whenever $\pi_\infty$ is a discrete series representation. In particular, Conjecture \ref{lapid-mao} is now known for all $\pi$ which have the property that $\pi_\infty$ is a discrete series representation.

For applications, one often needs a more explicit version of Conjecture \ref{lapid-mao}, which requires us to compute or quantify the quantities $J_0(g_{0,v} \cdot \phi_v)$ for $v \in S$. Chen and Ichino \cite{chen-ichino} proved the Lapid--Mao conjecture in such an explicit form under the following assumptions: $g_0=1$, $\pi$ is of squarefree conductor, and $\phi$ is the paramodular newvector at finite places and the vector of minimal weight at infinity. It is noteworthy that they did not compute $J_0(\phi_v)$ directly, but instead reduced to the endoscopic case and used the Rallis inner product formula.

Our next result assumes Conjecture \ref{lapid-mao} and gives an explicit formula for $\frac{|W_\phi(g_0) |^2   }{\langle \phi, \phi\rangle}$ in new cases.

\begin{theorem}\label{t:lapidmaoexplicit}
 Let $S_1$, $S_2$ and $S_3$ be disjoint, finite (possibly empty) sets of non-archimede\-an places of $\Q$. Let $S=S_1 \cup S_2 \cup S_3 \cup\{\infty\}$. Let $\pi\simeq\otimes_v \pi_v$ be an irreducible, unitary, cuspidal, generic automorphic representation of $G(\A)$ with trivial central character and let $\phi$ be a cusp form in the space of $\pi$ corresponding to $\otimes_v \phi_v$. We assume that $\pi$ and $\phi$ satisfy the following conditions.
 \begin{itemize}
  \item For each $p \in S_1$, $\pi_p$ is a simple supercuspidal representation and $\phi_p$ is a minimal vector in the space of $\pi_p$.
  \item For each $p \in S_2$, $\pi_p$ is a simple supercuspidal representation and $\phi_p$ is a local (paramodular) newvector in the space of $\pi_p$.
  \item     For each $p \in S_3$, $a(\pi_p)=1$ and $\phi_p$ is a local (paramodular) newvector in the space of~$\pi_p$.
  \item The representation $\pi_\infty$ is one of the following types:
   \begin{enumerate}
    \item (Large discrete series)  $\pi_\infty|\Sp_4(\R) = D_{(\lambda_1, \lambda_2)} \oplus D_{(-\lambda_2, -\lambda_1)}$ where $D_{(\lambda_1, \lambda_2)}$ is the (limit of) discrete series representation of $\Sp_4(\R)$ with Blattner parameter\linebreak ${(\lambda_1, \lambda_2) \in \Z^2}$ such that $1-\lambda_1 \le \lambda_2 \le 0$. In this case, $\phi_\infty$ is a lowest weight vector in the minimal $K_\infty$-type of $D_{(-\lambda_2, -\lambda_1)}$.
    \item (Principal series) $\pi_\infty|\Sp_4(\R) = \Ind_{B(\R) \cap \Sp_4(\R)}^{\Sp_4(\R)} (|\cdot|^\lambda_1 \boxtimes |\cdot|^\lambda_2)$ for some $\lambda_1, \lambda_2 \in \C$. In this case, $\phi_\infty$ is a $K_\infty$-fixed vector in the space of $\pi_\infty$.
   \end{enumerate}
  \item For $p \notin S$, the representation $\pi_p$ is unramified, and $\phi_p$ is the unique, up to scalars, spherical vector in $\pi_p$.
 \end{itemize}
 Let $g_0=\prod_{p \in S_1} g_{0,p}$ such that for each $p\in S_1$, $g_{0,p} = \begin{bsmallmatrix}\alpha_p^2\beta_p\\&\alpha_p\beta_p\\&&\alpha_p\\&&&1\end{bsmallmatrix}$ with $\alpha_p \in -p^{-1} + \Z_p$, $\beta_p \in -p^{-1} + \Z_p$. Assume Conjecture \ref{lapid-mao}.
 Then
 $$
  \frac{|W_\phi(g_0) |^2   }{\langle \phi, \phi\rangle} =  2^{-c}\frac{\zeta^{S} (2) \zeta^{(S)}(4) }{L^{S}(1,  \pi, \Ad)}  \left(\prod_{p \in S_1}p^7 \right)\left(\prod_{p \in S_2}p^5 \zeta_p(2) \right)\left(\prod_{p \in S_3}\frac{p \ \zeta_p(2)^2}{L(1, \pi_p, \Ad)} \right)J_\infty,
 $$
 where $c$ is as in Conjecture \ref{lapid-mao}, and
 $$
  J_\infty = \frac{|W_\infty(1)|^2}{L(1, \pi_\infty, \Ad)}\cdot
   \begin{cases}
    2^{\lambda_2 - \lambda_1-5}\pi^{\lambda_2 - 3 \lambda_1 - 8} (1+\lambda_1 - \lambda_2) &\text{\!if $\pi_\infty$ is in the large discrete series}, \\
    2^4 \pi^{-3}    &\text{\!if $\pi_\infty$ is in the principal series},
   \end{cases}
 $$
 with the function $W_\infty(1)$ defined as in Sect.~1.1 of \cite{chen-ichino}.
\end{theorem}
\begin{proof}
Since we are assuming Conjecture \ref{lapid-mao} we have $$\frac{|W_\phi(g_0) |^2   }{\langle \phi, \phi\rangle} =  2^{-c}\frac{\zeta^{S} (2) \zeta^{(S)}(4) }{L^{S}(1,  \pi, \Ad)}  \prod_{v \in S}J_0(g_{0,v} \cdot \phi_v).$$ Note that $g_{0,v} =1$ if $v \notin S_1$. Comparing Conjecture \ref{lapid-mao} and the main results of \cite{chen-ichino}  we see that
\begin{enumerate}
\item For each $p$ such that $a(\pi_p)=1$ and $\phi_p$ is a local newvector, we have $J_0(\phi_p) = \frac{p \ \zeta_p(2)^2}{L(1, \pi_p, \Ad)}$.

\item If $\pi_\infty$ is a large discrete series or principal series representation with parameters as in the theorem, and $\phi_\infty$ is a vector in $\pi_\infty$ as in the theorem, then $J_0(\phi_\infty) = J_\infty$.
\end{enumerate}
For $p \in S_1$ or $p \in S_2$, $J_0(g_{0,p} \cdot \phi_p)$ was computed in Propositions \ref{J0-non-vanishing} and \ref{J0-prop} respectively. Putting everything together, we obtain the desired result.
\end{proof}

\subsection{An application to lower bounds for sup-norms of newforms}\label{s:sup}

The results of the previous subsection allow us to provide a lower bound for sup-norms of global newforms with respect to the paramodular group. In fact, we show that newforms take ``large values" in the compact set $U(\Q)\bs U(\A)$.

More precisely, for each positive integer $N$, we define the compact subset $\F_N \subset U(\R)$ as follows:
$$
 \F_N = \left\{\begin{bsmallmatrix}1&a\\&1\\&&1&-a\\&&&1\end{bsmallmatrix}
 \begin{bsmallmatrix}1&&x&y\\&1&z&x\\&&1&\\&&&1\end{bsmallmatrix}
\in U(\R): 0 \le a,x,z \le 1, \  0 \le y \le 1/N \right\}.
$$
We have the following lemma.
\begin{lemma}Let $N= \prod_{p<\infty}p^{n_p}$ be a positive integer and let $u \in U(\A)$. Then we can write $u= u_\Q u_\R \prod_{p<\infty}u_p$ with $u_\Q  \in U(\Q)$, $u_\R \in \F_N$, $u_p \in K(p^{n_p})\cap U(\Q_p)$, where $K(p^{n_p}) \subset G(\Q_p)$ is the paramodular subgroup as defined in \eqref{paradefeq}.
\end{lemma}
\begin{proof} For elements $A, B \in U(\A)$, write $A \sim B$ if there exist $u_\Q  \in U(\Q)$ and $u_\f \in  \prod_{p<\infty}(K(p^{n_p})\cap U(\Q_p))$ such that $A = u_\Q B u_\f$. Let $u \in U(\A)$ be as in the statement. Then strong approximation for the unipotent group $U$ implies that $u \sim u_1$ for some $u_1 \in U(\R)$. We multiply $u_1$ on the left by a suitable element of $U(\Z)$ and on the right by the inverse of the finite part of the same element to conclude that $u_1 \sim u_2$ where $u_2 \in U(\R)$ is equal to $\begin{bsmallmatrix}1&a\\&1\\&&1&-a\\&&&1\end{bsmallmatrix}
 \begin{bsmallmatrix}1&&x&y\\&1&z&x\\&&1&\\&&&1\end{bsmallmatrix}$ for some $0 \le a,x,y,z \le 1$.
Finally, let $m \in \Z$ such that $0 \le y + \frac{m}{N} \le  \frac{1}{N}$. Put $k = \begin{bsmallmatrix}1&&&\frac{m}{N}\\&1&&\\&&1&\\&&&1\end{bsmallmatrix} \in U(\Q)$ and let $k_\f$ (resp.~$k_\infty$) be its image in $\prod_{p<\infty}U(\Q_p)$. Since $k_\f \in \prod_{p<\infty}(K(p^{n_p})\cap U(\Q_p))$, it follows that  $u\sim u_\R:=k_\infty u_2$ and we check that $u_\R$ has the required properties.
\end{proof}

The significance of the above proposition is that for any bounded automorphic form $\phi$ on $G(\Q)\bs G(\A)$ that is right invariant by $\prod_{p<\infty}K(p^{n_p})$, we have  $$\sup_{g \in U(\Q) \bs U(\A)}\frac{|\phi(g)|}{\langle \phi, \phi \rangle^{1/2}} =  \sup_{g_\infty \in \F_N}\frac{|\phi(g_\infty)|}{\langle \phi, \phi \rangle^{1/2}}.$$

\begin{theorem}\label{t:largevalues}
 Let $\pi\simeq\otimes_v \pi_v$ be an irreducible, unitary, cuspidal, generic automorphic representation of $G(\A)$ with trivial central character and conductor $N = \prod p^{a(\pi_p)}$. Let $\phi$ in the space of $\pi$, corresponding to $\otimes_v \phi_v$, be such that $\phi_p$ is a newvector with respect to the paramodular subgroup $K(p^{a(\pi_p)})$ at each prime $p$. Assume that at each prime $p|N$, $\pi_p$ is either a simple supercuspidal representation (so that $a(\pi_p)=5$), or is a representation satisfying $a(\pi_p)=1$ (so that $\pi_p$ is of type IIa in the notation of \cite{NF}). Assume also that Conjecture \ref{lapid-mao} is true and that $\pi_\infty$ and $\phi_\infty$ are among the types covered by Theorem \ref{t:lapidmaoexplicit}. Then
 $$
  \sup_{g \in U(\Q) \bs U(\A)}\frac{|\phi(g)|}{\langle \phi, \phi \rangle^{1/2}} =  \sup_{g_\infty \in \F_N}\frac{|\phi(g_\infty)|}{\langle \phi, \phi \rangle^{1/2}} \gg_{\pi_\infty, \eps} N^{1/2 - \eps}
 $$
\end{theorem}
\begin{proof}
Using the definition \eqref{global-whittaker-defn} of $W_\phi(g)$ we get
\begin{align*}
 \frac{\sup_{g\in G(\A)}|\phi(g)|}{\langle \phi, \phi \rangle^{1/2}} & \ge \vol(U(\Q) \bs U(\A))^{-1}\frac{|W_{\phi}(1)|}{\langle \phi, \phi \rangle^{1/2}} \\
 &\asymp  \left(\frac{1 }{L^{S}(1,  \pi, \Ad)} \right)^{1/2} \left(\prod_{p \in S_2}p^5 \right)^{1/2}\left(\prod_{p \in S_3}\frac{p \ }{L(1, \pi_p, \Ad)} \right)^{1/2}J_\infty^{1/2},
\end{align*}
where in the last step we have used Theorem~\ref{t:lapidmaoexplicit}.

We know that $\pi$ lifts to a unitary automorphic representation $\Pi$ of $\GL_4(\A)$ which is either cuspidal or an isobaric sum of two cuspidal representations of $\GL_2(\A)$; we have $L(s, \pi, \Ad) = L(s, \Pi, \rm{Sym}^2)$. Using the fact that $L(s, \Pi, \rm{Sym}^2)$ has no pole at $s=1$, it follows from the main result of \cite{Li} that $L^S(1, \pi, \Ad) \ll_{\eps, \pi_\infty} N^\eps$. Clearly, $J_\infty \gg_{\pi_\infty} 1$. The result follows.
\end{proof}

We remind the reader that if $\pi_\infty$ is in the discrete series, Conjecture 4.1 is known by recent work of Furusawa and Morimoto \cite[Theorem 6.3 and Corollary 8.1]{FM22}.

\subsection{Explicit global Novodvorsky integral}
Let $\pi\simeq\otimes_v \pi_v$ be an irreducible, unitary, cuspidal, globally generic automorphic representation of $G(\A)$ with trivial central character. Let $\chi = \otimes_v \chi_v$ be a unitary Hecke character of $\Q^\times \bs \A^\times$. Let $\psi$ be the character of $\Q \bs \A$ given in Sect.~\ref{notations-sect}. For $\phi\in\pi$, corresponding to $\otimes_v \phi_v$, define the global Novodvorsky integral by
\begin{equation}\label{global-novo-eqn}
Z(s, \phi, \chi) := \int\limits_{\Q^\times \bs \A^\times} \int\limits_{(\Q \bs \A)^3} \phi(\begin{bsmallmatrix}1&x_2&&x_4\\&1\\&z&1&-x_2\\&&&1\end{bsmallmatrix} \begin{bsmallmatrix}y\\&y\\&&1\\&&&1\end{bsmallmatrix})\chi(y)\psi(x_2)|y|^{s-\frac 12}\,dz\,dx_2\,dx_4\,d^\times y.
\end{equation}
Let $W_\phi$ be the global Whittaker function corresponding to $\phi$ as given in (\ref{global-whittaker-defn}). Then, a standard unfolding process (see Chapter 3 of \cite{Bump87}) gives
$$
 Z(s, \phi, \chi) = \int\limits_{\A^\times} \int\limits_\A W_\phi(\begin{bsmallmatrix}y\\&y\\&x&1\\&&&1\end{bsmallmatrix}) \chi(y) |y|^{s-\frac 32}\,dx\,d^\times y.
$$
Let  $W_v \in \mathcal{W}(\pi_v, (\psi_v)_{-1, -1})$ correspond to $\phi_v$ in the local Whittaker model. Then by uniqueness of Whittaker functionals we have the basic identity (see \cite{Bump87})
$$\frac{Z(s, \phi, \chi)}{W_\phi(g)} = \prod\limits_v \frac{Z_v(s, W_v, \chi_v)}{W_v(g_v)},$$
where $Z_v(s, W_v, \chi_v)$ is defined in (\ref{zeta-int-defn}) and $g = (g_v) \in G(\A)$ is any element such that ${W_\phi(g) \neq 0}$. We have the following theorem.
\begin{theorem}\label{global-novo-thm}
 Let $\pi\simeq\otimes_v \pi_v$ be an irreducible, unitary, cuspidal, globally generic automorphic representation of $G(\A)$ with trivial central character. Let $\chi = \otimes_v \chi_v$ be a unitary Hecke character on $\Q^\times \bs \A^\times$. Let $\phi\in\pi$ correspond to $\otimes_v \phi_v$. Let $S$ be a finite (possibly empty) set of prime numbers. We assume that $\pi$, $\chi$ and $\phi$ satisfy the following conditions.
\begin{itemize}
\item For each $p \in S$,  $a(\chi_p) \le 1$, $\pi_p$ is a simple supercuspidal representation of $G(\Q_p)$,  and $\phi_p$ is the translate of the minimal vector of $\pi_p$ by ${\rm diag}(-p^{-3}, p^{-2}, -p^{-1},1)$.
\item We have $\chi_\infty$ equal to the trivial character, and  $\pi_\infty|_{\Sp_4(\R)} = D_{(\lambda_1, \lambda_2)} \oplus D_{(-\lambda_2, -\lambda_1)}$ where $D_{(\lambda_1, \lambda_2)}$ is the (limit of) discrete series representation of $\Sp_4(\R)$ with Blattner parameter $(\lambda_1, \lambda_2) \in \Z^2$ such that $1-\lambda_1 \le \lambda_2 \le 0$. The vector $\phi_\infty$ is a lowest weight vector in the minimal $K_\infty$-type of $D_{(-\lambda_2, -\lambda_1)}$.

\item For $p \notin S$, the representation $\pi_p$ and the character $\chi_p$  are unramified, and $\phi_p$ is the unique (up to scalar multiples) spherical vector in $\pi_p$.

\end{itemize}
 Then we have
$$\frac{Z(s, \phi, \chi)}{W_\phi(1)} = L(s, \pi \times \chi) \frac{Z_\infty(s, W_\infty)}{L(s, \pi_\infty) W_\infty(1)}  \prod\limits_{p \in S} \frac{1}{(1-p^{-1})p^{3}},$$
where $W_\infty(1)$ and $Z_\infty(s, W_\infty)$ is as given in Proposition 7.1 ii) and Proposition 8 of \cite{Mor04}.
\end{theorem}
\begin{proof}
By Proposition 3.9 of \cite{Ram-Tak}, we know that $L(s, \pi_p \times \chi_p) = 1$ for all $p \in S$. The result now follows from \eqref{e:localspinunram}, Corollary \ref{Whittaker-model-defn-cor} and Proposition \ref{zeta-int-value-prop}.
\end{proof}

Using Theorem \ref{t:lapidmaoexplicit}, we can write the above theorem in an equivalent form with the factor $W_\phi(1)$ replaced by the Petersson norm, which is more suited for analytic applications. We give a simplified version of this result in the next corollary.
\begin{corollary}\label{c:globalnov}
 Let $\pi$, $\phi$, $\chi$ and $S$ be as in Theorem \ref{global-novo-thm}. Then
$$\frac{|Z(s, \phi, \chi)|}{\langle \phi, \phi \rangle^{1/2}} = C_\infty(s) \frac{|L(s, \pi \times \chi)|}{\sqrt{L(1, \pi, \Ad)}} \  \left(\prod\limits_{p \in S}p^{1/2}\sqrt{ \frac{\zeta_p(1)^2 L(1, \pi_p, \Ad)}{\zeta_p(2)\zeta_p(4)}}\right),$$
where $C_\infty(s)$ depends only on $\pi_\infty$ and $s$.
\end{corollary}

\subsection{Explicit Gan-Gross-Prasad conjecture for \texorpdfstring{$(\SO(5), \SO(2))$}{}}
In this section, we will write down an explicit version of the Gan--Gross--Prasad conjecture for $(\SO(5), \SO(2))$ (which is now a theorem due to Furusawa and Morimoto) in new cases. Let $L$ be an imaginary quadratic extension of $\Q$ with discriminant $-D$, and let $S$ be defined by
$$
 S := \begin{cases} \mat{1}{}{}{D/4} & \text{ if } D \equiv 0 \pmod{4}, \\[2ex]
 \mat{1}{1/2}{1/2}{(1+D)/4} & \text{ if } D \equiv -1 \pmod{4}.\end{cases}
$$
Let $T_S := \{ g \in \GL_2 : {}^tg S g = \det(g) S\}$. We see that $T_S(\Q) \simeq L^\times$. Let $\Lambda$ be a character of $L^\times \bs \A_L^\times$ that is trivial on $\A^\times$, and consider it as a character on $T_S(\A)$. Embed $T_S$ in $G$ by
$$ T_S \ni g \mapsto \mat{g}{}{}{{\rm det}(g) g'} \text{ where } g' := \mat{}{1}{1}{} {}^tg^{-1} \mat{}{1}{1}{}.
$$
Let $N$ be the unipotent radical of the Siegel parabolic subgroup of $G$, and let $\theta_S$ be the character of $N(\A)$ given by
$$\theta_S(\mat{1}{X}{}{1}) := \psi({\rm tr}(SX\mat{}{1}{1}{})).$$
Here, $\psi$ is the character of $\Q \bs \A$ given in Sect.~\ref{notations-sect}. Now, let $\pi$ be an irreducible cuspidal automorphic representation of $\GSp_4(\A)$ with trivial central character. For any $\phi \in V_\pi$, define the global Bessel period by
\begin{equation}\label{defbessel}
  B(\phi, \Lambda) =
  \int\limits_{\A^\times T_S(\Q)\bs T_S(\A)}\;\int\limits_{N(\Q) \bs N(\A)}\phi(tn)\Lambda^{-1}(t) \theta_S^{-1}(n)\,dn\,dt
\end{equation}
where we use the Tamagawa measure. For each place $v$, fix a $G(\Q_v)$-invariant Hermitian inner product $\langle\, , \rangle_v$ on $\pi_v$. For $\phi_v \in V_{\pi_v}$, define
\begin{equation}\label{e:J_vdeffinal}
 J_v(\phi_v ) =\frac{L(1, \pi_v, \Ad)L(1, \chi_{d,v})\int_{\Q_v^\times \bs T(\Q_v)}\int_{N(\Q_v)}\frac{\langle \pi_v(t_vn_v) \phi_v , \phi_v \rangle}{\langle \phi_v , \phi_v\rangle} \Lambda_v^{-1}(t_v) \theta_S^{-1}(n_v)\,dn_v\,dt_v}{\zeta_{\Q_v}(2)\zeta_{\Q_v}(4)L(1/2, \pi_v \otimes \AI(\Lambda_v^{-1}))}.\end{equation}
Strictly speaking, the integral above may not converge absolutely, in which case one defines it via regularization (see~\cite[p.~6]{yifengliu}). It can be shown that $J_v(\phi_v )=1$ for almost all places. We now state the refined conjecture as phrased by Liu~\cite{yifengliu}.

\begin{conjecture}[Yifeng Liu]\label{c:liu}
 Let $\pi$, $\Lambda$ be as above. Suppose that for almost all places $v$ of $\Q$ the local representation $\pi_v$ is generic. Let $\phi$ be an automorphic form in the space of~$\pi$ corresponding to $\otimes_v \phi_v$. Then
 \begin{equation}\label{e:refggpnew}
  \frac{|B(\phi, \Lambda)|^2}{\langle \phi, \phi \rangle} = \frac{C_T}{S_\pi}\frac{\zeta_\Q(2)\zeta_\Q(4)L(1/2, \pi \times \AI(\Lambda^{-1}))}{L(1, \pi, \Ad)L(1, \chi_{-D})} \prod_v J_v(\phi_v ),
 \end{equation}
 where $\zeta_\Q(s) = \pi^{-s/2}\Gamma(s/2)\zeta(s)$ denotes the completed Riemann zeta function, $C_T$ is a constant relating our choice of local and global Haar measures, and $S_\pi$ denotes a certain integral power of 2, related to the Arthur parameter of $\pi$. In particular,
 $$
  S_\pi = \begin{cases} 4 &\text{ if } \pi \text{ is endoscopic,} \\ 2 &\text{ if }\pi \text{ is of general type.} \end{cases}
 $$
\end{conjecture}
 Recently Furusawa and Morimoto proved Conjecture \ref{c:liu} for tempered $\pi$; see \cite{FM21}, \cite{FM22}.

For several applications it is important  to have an explicit formula for the right hand side of (\ref{e:refggp}), which amounts to computing the local integrals $J_v(\phi_v )$ for appropriate choices of $\phi_v$. In \cite{DPSS15}, we computed $J_v(\phi_v)$ if $v$ is non-archimedean and $\pi_v$ has a non-zero $P_1$-fixed vector $\phi_v$, and if $v = \infty$ and $\pi_\infty$ is a holomorphic discrete series representation with minimal scalar $K_\infty$ type. Here, the congruence subgroup $P_1$ is given by
\begin{equation}\label{P1-defn}
P_1 := \{ g = \mat{A}{B}{C}{D} \in G(\Z_p) : C \equiv 0 \pmod{p}\}.
\end{equation}
In \cite{FM22}, the authors extend the explicit computations to include general holomorphic discrete series representations $\pi_\infty$ with not necessarily minimal scalar $K_\infty$ type. Below, we use the local computation from Sect.~\ref{locbessec} to obtain the explicit formula when we allow $\pi_v$ to be a simple supercuspidal representation.
\begin{theorem}\label{Bessel-per-thm}
 Let $L$ be an imaginary quadratic extension of $\Q$ with discriminant $-D$, and let $\Lambda$ be a character of $L^\times \bs \A_L^\times$ that is trivial on $\A^\times$.   Furthermore, assume that $\Lambda_\infty$ is trivial. Let $\pi\simeq\otimes_v \pi_v$ be an irreducible, unitary, cuspidal, non-CAP automorphic representation of $G(\A)$ with trivial central character and let $\phi$ be a cusp form in the space of $\pi$ corresponding to $\otimes_v \phi_v$. Let $S$ be a finite (possibly empty) set of odd prime numbers.

 We assume that $\pi$ and $\phi$ satisfy the following conditions.
\begin{itemize}
\item For each $p \in S$, $L_p$ is an inert field extension of $\Q_p$,  $\pi_p$ is a simple supercuspidal representation of $G(\Q_p)$ and $\Lambda_p$ satisfies the property that $m_p := {\rm min}\{ m \geq 0 : \Lambda_p|_{1+p^m\Z_p[\sqrt{-D}]} = 1\} \geq 2$. Moreover, let $u_p$ be defined by (\ref{u0def}), $\alpha_p = p^{1-m_p}u_p$, $\beta_p = p$, $A_p = \mat{}{1}{1}{}$ if $4 | D$,  $A_p = \mat{-1/2}{1}{1}{}$ if $p \nmid D$ and set $g_p = \mat{A_p}{}{}{A_p'} d_{\alpha_p, \beta_p}^{-1}$. Then
 $\phi_p$ is the translate of the minimal vector $\phi_{\rm min}$ in the space of $\pi_p$ by the matrix $g_p$.
\item The representation $\pi_\infty$ has scalar minimal $K_\infty$ type $(k, k)$ with $k \geq 2$. The vector $\phi_\infty$ spans this one-dimensional $K_\infty$ type. If $k=2$, then $\pi$ is tempered.
\item For $p \notin S$, a finite prime, the representation $\pi_p$ is unramified, and $\phi_p$ is the unique, up to scalars, spherical vector in $\pi_p$.
\end{itemize}
Then
$$ \frac{|B(\phi, \Lambda)|^2}{\langle \phi, \phi \rangle}  = e^{-4\pi {\rm Tr}(S)} D^{k-2} 2^{2k-c} \frac{\zeta_\Q(2)\zeta_\Q(4)L^S(1/2, \pi \times \AI(\Lambda^{-1}))}{L^S(1, \pi, \Ad)L(1, \chi_{-D})^2} \prod\limits_{p \in S} J_p,$$
where $c = 1$ if $\pi$ is endoscopic, and $c=0$ if $\pi$ is of general type. For $p \in S$, we have
$$J_p = \frac{(1-p^{-1})(1-p^{-4})}{1+p^{-1}}p^{-4m_p+7}.$$
\end{theorem}
\begin{proof}
For the case that $\pi$ is tempered, the conjectural formula (\ref{e:refggp}) is proven in Theorem~1.2 of \cite{FM22}. For $k > 3$, it is known that a non-CAP $\pi$ is tempered by Proposition~8.1 of \cite{FM22}. The value of $J_\infty$ has been computed in \cite[Sec. 3.5]{DPSS15} and in particular we have $$C_TJ_\infty = \frac{2^{2k+1}D^{k-2} e^{-4\pi {\rm Tr}(S)}}{L(1, \chi_{-D})}.$$ The values of $J_p$ for $p \in S$ follow from Proposition \ref{Bessel-model-prop}, the remark about the volume of $\OF^\times \bs \OF_L^\times$ before that proposition, and \eqref{e:J_vdeffinal}. Now the theorem follows by substituting these quantities into (\ref{e:refggp}).
\end{proof}

\bibliography{simple_analytic}{}
\bibliographystyle{plain}

\end{document}